\def\ORG{TT}
\if\ORG
\documentclass[dvipdfmx,a4paper]{amsproc}
\usepackage{amsmath, amssymb, amsfonts}
\usepackage[dvipdfm]{graphicx}
\usepackage{tikz}
\theoremstyle{plain}
 \newtheorem{thm}{Theorem}[section]
 \newtheorem{theorem}{Theorem}[section]      % new
 
 \newtheorem{lemma}[thm]{Lemma}              % lem
 \newtheorem{proposition}[thm]{Proposition}  % prop
\theoremstyle{definition}
 \newtheorem{definition}[thm]{Definition}    % defn

\theoremstyle{remark}
 \newtheorem{remark}[thm]{Remark}            % rem

\DeclareMathOperator{\RE}{Re}
\DeclareMathOperator{\Ad}{Ad}
\def\p{\partial}
\def\D{\mathrm{d}}
\begin{document}

\title[Integral transformations of hypergeometric functions]
{Integral transformations of hypergeometric functions with several variables}

\author{Toshio \textsc{Oshima}}% Author Name (\sc should NOT be used here)
\address{Center for Mathematics and Data Science, Josai University, 2-3-20 Hirakawacho, Chiyodaku, Tokyo 102-0093, Japan}\email{oshima@ms.u-tokyo.ac.jp}

\thanks{%The author thanks to the referees who suggest several improvement of expressions
%in the first manuscript. 
This work was supported by Grant-in-Aid for Scientific Researches (C), 
No.\ 18K03341, Japan Society of Promotion of Science}

% Subject code(s)
% Key word(s)

\begin{abstract}
As a generalization of Riemann-Liouville integral, 
we introduce integral transformations of convergent power series which can be applied to 
hypergeometric functions with several variables. 
\end{abstract}

\maketitle
\else
%%%%%%%%%%%%%%%%%%%% author.tex %%%%%%%%%%%%%%%%%%%%%%%%%%%%%%%%%%%
%
% sample root file for your "contribution" to a contributed volume
%
% Use this file as a template for your own input.
%
%%%%%%%%%%%%%%%% Springer %%%%%%%%%%%%%%%%%%%%%%%%%%%%%%%%%%

% RECOMMENDED %%%%%%%%%%%%%%%%%%%%%%%%%%%%%%%%%%%%%%%%%%%%%%%%%%%
\documentclass[dvipdfmx,graybox]{svmult}

% choose options for [] as required from the list
% in the Reference Guide
\usepackage{tikz}

\usepackage{mathptmx}       % selects Times Roman as basic font     <-  \boldsymbol
\usepackage{helvet}         % selects Helvetica as sans-serif font
\usepackage{courier}        % selects Courier as typewriter font
\usepackage{type1cm}        % activate if the above 3 fonts are
                            % not available on your system
%
\usepackage{makeidx}        % allows index generation
\usepackage{graphicx}       % standard LaTeX graphics tool
                            % when including figure files
\usepackage{multicol}       % used for the two-column index
\usepackage[bottom]{footmisc}% places footnotes at page bottom

% see the list of further useful packages
% in the Reference Guide
\usepackage{amsmath,amssymb}
\usepackage{bm}

\DeclareMathOperator{\RE}{Re}
\DeclareMathOperator{\Ad}{Ad}

\def\p{\partial}
\def\boldsymbol{\bm}

\makeindex             % used for the subject index
                       % please use the style svind.ist with
                       % your makeindex program

%%%%%%%%%%%%%%%%%%%%%%%%%%%%%%%%%%%%%%%%%%%%%%%%%%%%%%%%%%%%%%%%%%%%%%%%%%%%%%%%%%%%%%%%%

\begin{document}

\title*{Integral transformations of hypergeometric functions with several variables}
%\titlerunning{Integral transformations of hypergeometric functions}
% Use \titlerunning{Short Title} for an abbreviated version of
% your contribution title if the original one is too long
\author{Toshio Oshima}
% Use \authorrunning{Short Title} for an abbreviated version of
% your contribution title if the original one is too long
\institute{Toshio Oshima \at Josai University, 2-2-20 Hirakawacho, Chiyodaku, 
Tokyo 102--0093, Japan, \\ \email{oshima@ms.u-tokyo.ac.jp}}
%
% Use the package "url.sty" to avoid
% problems with special characters
% used in your e-mail or web address
%
\maketitle

\abstract*{
As a generalization of Riemann-Liouville integral, 
we introduce integral transformations of convergent power series
which can be applied to hypergeometric functions with several variables. 
}
\abstract{
As a generalization of Riemann-Liouville integral, 
we introduce integral transformations of convergent power series
which can be applied to hypergeometric functions with several variables. 
}
\fi

\section{Introduction}
Suppose a function $\phi(x)$ satisfies a linear ordinary differential equation on $\mathbb P^1$. 
Then the Riemann-Liouville integral \eqref{eq:RL} of $\phi(x)$ induces a middle convolution of 
the differential equation defined by Katz \cite{Kz}.
The multiplication of $\phi(x)$ by a simple function $(x-c)^\lambda$ induces an addition of the
differential equation which is also important.  For example, any rigid irreducible linear 
Fuchsian differential equation is constructed by successive applications of middle convolutions 
and additions.  Hence we have an integral representation of its solution, which is shown 
first by Katz \cite{Kz} in the case of Fuchsian systems of the first order 
and by the author \cite{Ow} in the case of single differential equations of higher orders.
Here the equation is called rigid if it is free from accessory parameters, namely, the equation 
is globally determined by the local structure at the singular points. 
% The equation satisfied by Gauss hypergeometric function or a generalized hypergeometric function
% is rigid. The function satisfied a rigid Fuchsian differential equation
Applying these transformations to linear ordinary differential equations on $\mathbb P^1$, we study
many fundamental problems on their solutions in \cite{Ow}.

The rigid Fuchsian ordinary differential equation %with more than three singular points 
on $\mathbb P^1$ can be extended to a Knizhnik-Zamolodchikov type equation (KZ equation in short, 
cf.~\cite{KZ}) regarding the singular points as new variables.  
Haraoka \cite{Ha} shows this by extending middle convolutions on KZ equations and its  
generalization for equations with irregular singularities is given by the author \cite{Ojm,Ov}.
Then these transformations are also useful to %certain classes of 
hypergeometric functions with several variables including Appell's hypergeometric functions (cf.~\cite{Okz}).

These transformations do not give an integral representation of Appell's hypergeometric series 
$F_4$ but K.~Aomoto gives an integral representation of $F_4$, which is written in 
\cite[\S13.10.2]{Ow}.  We define integral transformations on the space of convergent power series 
of several variables in \S2 and study hypergeometric functions with several variables, which 
extends a brief study of Appell's hypergeometric functions in \cite[\S3.10]{Ow}.
The transformations are invertible and they are generalizations of Riemann-Liouville 
integrals of functions with a single variable.
On the space of hypergeometric functions with several variables we have important 
transformations such as the multiplications of suitable functions and coordinate transformations.
Note that a holonomic Fuchsian differential equation of several variables has solutions of 
convergent power series times simple functions $x_1^{\lambda_1}\cdots x_n^{\lambda_n}$
% of powers of coordinate functions, 
%, namely,  
%$x_1^{\lambda_1}\cdots x_n^{\lambda_n}$ of powers of coordinate functions, 
at its normally crossing singular points (cf.~\cite{KO}). 
We study a combination of the integral transformations and multiplications by these 
simple functions.

In \S3 we show that the transformations defined in \S2 give integral representations of Appell's 
or Lauricella's  hypergeometric series and certain Horn's hypergeometric series with irregular 
singularities.

In \S4 we show that our study gives a result related to the connection problem on the solutions,
which will be discussed in another paper and related to the study by Matsubara \cite[\S3]{Ma}. 

In \S5 a combination of the integral transformations with coordinate 
transformations defined by products of powers of coordinate functions parametrized by 
$GL(n,\mathbb Z)$. The transformations given in \S5 are related to $A$-hypergeometric
series introduced by Gel'fand, Kpranov and Zelevinsky \cite{GKZ}. The transformation
\[
  \sum_{m=0}^\infty\sum_{n=0}^\infty c_{m,n}x^my^n\mapsto
   \sum_{m=0}^\infty\sum_{n=0}^\infty c_{m,n}\frac{(\alpha)_{p_1m+q_1n}(\beta)_{p_2m+q_2n}}
  {(\gamma)_{(p_1+p_2)_m+(q_1+q_2)_n}}x^my^n
\]
of convergent power series is an example.  Here
$p_1$, $p_2$, $q_1$ and $q_2$ are non-negative integers with 
$\left(\begin{smallmatrix}p_1&p_2\\ q_1&q_2\end{smallmatrix}\right)\in GL(2,\mathbb Z)$
and we put $(a)_k=a(a+1)\cdots (a+k-1)$. 

In \S6 we study the transformation of differential equations corresponding to the transformations
of their solutions.

In \S7 we study the transformation given in \S5 which keeps the space of KZ equations 
\begin{align*}
  \mathcal M : \begin{cases}
   \displaystyle\frac{\p u}{\p x}=
   \frac{A_{01}}{x-y}u+\frac{A_{02}}{x-1} u + \frac{A_{03}}{x} u,\\
   \displaystyle\frac{\p u}{\p y}=
   \frac{A_{01}}{y-x}u+\frac{A_{12}}{y-1} u + \frac{A_{13}}{y} u
  \end{cases}
\end{align*}
and give the induced transformations of the residue matrices $A_{i,j}$ defining the equations. 
The transformation is reduced to
a coordinate transformation corresponding to the coordinate symmetries described in 
\cite[\S6]{Okz} and a middle convolution of the KZ equation. 
Hence we apply the result in \cite{Ha,Okz} to them.
The hypergeometric series 
\begin{equation*}
\sum_{m=0}^\infty\sum_{n=0}^\infty
  \frac{\prod_{i=1}^p (\alpha_i)_{m}\prod_{j=1}^q (\beta_j)_{n}\prod_{k=1}^r (\gamma_k)_{m+n}}
  {\prod_{i=1}^{p} (1-\alpha'_i)_{m}\prod_{j=1}^{q} (1-\beta'_j)_n\prod _{k=1}^{r}
   (1-\gamma'_k)_{m+n}}x^my^n
\end{equation*}
with $\alpha'_1=\beta'_1=0$ is a typical example satisfying  a KZ equation, which is a 
generalization of Appell's $F_1$.

To the KZ equation we show Theorem \ref{thm:KZ2} which gives
an interesting correspondence between \emph{simple} solutions along a
line (cf.~Definition~\ref{def:simple}) and \emph{simple} solutions at a singular point where 
three singular lines meet.

In \S8 we restrict our transformations in \S7 to certain ordinary differential 
equations of Shlesinger canonical form. 
The transformation may be interesting since it may change the index of rigidity defined by 
\cite{Kz}.

Several applications of the results in this paper will be given in other papers.

%%%%%%%%%%%  Integral transformations %%%%%%%%%%%%%
\section{Integral transformations}\label{sec:Int}
The Riemann-Liouville transform $I_c^\mu \phi$ of a function $\phi(x)$ is defined by
\begin{equation}\label{eq:RL}
 (I_{x,c}^\mu\phi)(x)=(I_c^\mu\phi)(x):=\frac1{\Gamma(\mu)}\int_c^x\phi(t)(x-t)^{\mu-1}\D t.
\end{equation}
Here $c$ is usually a singular point of $\phi(x)$. 
Since
\begin{align*}
 \int_0^x\phi(t)(x-t)^{\mu-1}\D t
  &=\int_0^1\phi(xs)(x-xs)^{\mu-1}x\D s\qquad(t=xs)\\
  &=x^\mu\int_0^1\phi(sx)(1-s)^{\mu-1}\D s,
\end{align*}
the transformation
\begin{equation}
 (K_x^\mu \phi)(x):=\frac1{\Gamma(\mu)}\int_0^1\phi(tx)(1-t)^{\mu-1}\D t
\end{equation}
of a function $\phi(x)$ satisfies
\begin{align}
  K_x^\mu &= x^{-\mu}I_0^\mu,\label{eq:K2I}\\
  K_x^\mu x^\alpha&=\frac{\Gamma(\alpha+1)}{\Gamma(\alpha+\mu+1)}x^\alpha.
\end{align}
%%%%%% Def K_x^\mu %%%%%%
\begin{definition}
We extend the integral transform $K_x^\mu$ to a function $\phi(x)$ of several variables 
$x=(x_1,\dots,x_n)$ by 
\begin{equation*}
 K_x^\mu \phi(x):=\frac1{\Gamma(\mu)}\int_{\substack{t_1>0,\dots,t_n>0\\t_1+\cdots+t_n<1}}
  (1-t_1-\dots-t_n)^{\mu-1} \phi(t_1x_1,\dots,t_nx_n)\D t_1\dots \D t_n.
\end{equation*}
\end{definition}
%%%
Note that 
\begin{align*}
\int_0^{1-s}t^{\alpha}(1-s-t)^{\mu-1}\D t&=\int_0^{1-s}t^{\alpha}(1-s)^{\mu-1}(1-\tfrac t{1-s})^{\mu-1}\D t\allowdisplaybreaks\\
&=(1-s)^{\alpha+\mu}\int_0^1 t^{\alpha}(1-t)^{\mu-1}\D t\\
&=\frac{\Gamma(\alpha+1)\Gamma(\mu)}{\Gamma(\alpha+\mu+1)}(1-s)^{\alpha+\mu}
\end{align*}
and hence
\begin{align*}
&\int_{\substack{t_1>0,\dots,t_n>0\\t_1+\cdots+t_n<1}}t^{\alpha_1}\cdots t^{\alpha_n}
(1-t_1-\cdots-t_n)^{\mu-1}\D t\\
&\ \ =\int_0^1\! t_1^{\alpha_1}\D t_1\int_0^{1-t_1}\!t_2^{\alpha_2}\D t_2\cdots\!\int_0^{1-t_1-\cdots-t_{n-1}}\! t_n^{\alpha_n}(1-t_1-\cdots-t_n)^{\mu-1}\D t_n\\
&\ \ =\frac{\Gamma(\mu)\Gamma(\alpha_n+1)}{\Gamma(\alpha_n+\mu+1)}\int_0^1\!\! t_1^{\alpha_1}\D t_1
\cdots\!\int_0^{1-t_1-\cdots-t_{n-2}}\!\! t_{n-1}^{\alpha_{n-1}}(1-t_1-\cdots-t_{n-1})^{\alpha_n+\mu}\D t_{n-1}
\\&\ \ 
 =\frac{\Gamma(\mu)\Gamma(\alpha_n+1)}{\Gamma(\alpha_n+\mu+1)}\times
  \frac{\Gamma(\alpha_n+\mu+1)\Gamma(\alpha_{n-1}+1)}{\Gamma(\alpha_{n-1}+\alpha_n+\mu+2)}
  \times\cdots \\
&\qquad\cdots\times
  \frac{\Gamma(\alpha_2+\cdots+\alpha_n+\mu+n-1)\Gamma(\alpha_1+1)}
  {\Gamma(\alpha_1+\cdots+\alpha_n+\mu+n)}\\
&\ \ =\frac{\Gamma(\mu)\Gamma(\alpha_1+1)\cdots\Gamma(\alpha_n+1)}
  {\Gamma(\alpha_1+\cdots+\alpha_n+\mu+n)}.
\end{align*}
Therefore we have
\begin{equation}\label{eq:Kp}
 K_x^\mu x^{\boldsymbol \alpha} = \frac{\Gamma(\boldsymbol\alpha+1)}
  {\Gamma(|\boldsymbol\alpha+1|+\mu)}x^{\boldsymbol \alpha}
\end{equation}
and
\begin{equation}\label{eq:Ks}\\
\begin{split}
 &K_x^\mu \phi(x_1)x_2^{\alpha_2-1}\cdots x_n^{\alpha_n-1}\\
 &\quad = 
\frac{\Gamma(\alpha_2)\cdots\Gamma(\alpha_n)}{\Gamma(\alpha_2+\cdots+\alpha_n+\mu)}
(I_0^{\alpha_2+\cdots+\alpha_n+\mu}\phi)(x_1)\cdot x_2^{\alpha_2-1}\cdots x_n^{\alpha_n-1}
\end{split}
\end{equation}
Here and hereafter we use the notation
\begin{equation}
\begin{split}
  \mathbb N&=\{0,1,2,,\ldots\},\\
 \mathbf m&\ge 0\ \Leftrightarrow \ m_1\ge0,\ldots,m_n\ge0,\\
 |\boldsymbol\alpha|&=\alpha_1+\cdots+\alpha_n,\ 
 \boldsymbol\alpha+c=(\alpha_1+c,\dots,\alpha_n+c),\ \mathbf m!=m_1!\cdots m_n!,
 \\
 x^{\boldsymbol \alpha}&=\mathbf x^{\boldsymbol \alpha}=x_1^{\alpha_1}\cdots x_n^{\alpha_n},\ 
 (c-\mathbf x)^{\boldsymbol \alpha}=(c-x_1)^{\alpha_1}\cdots (c-x_n)^{\alpha_n},\\
 \Gamma(\boldsymbol \alpha)&=\Gamma(\alpha_1)\cdots\Gamma(\alpha_n),\ 
 (\boldsymbol \alpha)_{\mathbf m}=\frac{\Gamma(\boldsymbol \alpha+\mathbf m)}
  {\Gamma(\boldsymbol \alpha)}
\end{split}\label{notation}
\end{equation}
for $\boldsymbol \alpha=(\alpha_1,\dots,\alpha_n)\in\mathbb C^n$, 
$\mathbf m=(m_1,\dots,m_n)\in\mathbb N^n$ and $x=(x_1,\dots,x_n)$.

The arguments in this section are valid when $\RE\alpha_1>0,\dots,\RE\alpha_n>0$ and $\RE\mu>0$ 
but the right hand side of \eqref{eq:Kp} is meromorphic for $\boldsymbol\alpha$ and $\mu$ and 
we define $K_x^\mu x^{\boldsymbol \alpha}$ by the analytic continuation with respect to these parameters.

We will define the inverse $L_x^\mu$ of $K_x^\mu$. 
Suppose $0\le \RE s<1$ and $0<c<1-\RE s$. Then
\begin{align*}
 &\int_{c-i\infty}^{c+i\infty}t^{-\alpha} (1-s-t)^{-\tau}\tfrac{\D t}t=
  (1-s)^{-\tau}\int_{c-i\infty}^{c+i\infty}t^{-\alpha}(1-\tfrac t{1-s})^{-\tau}\tfrac{\D t}t
  \phantom{AAAAAAA}
\\
 &\quad=(1-s)^{-\alpha-\tau}\int_{\tfrac{c-i\infty}{1-s}}^{\tfrac{c+i\infty}{1-s}}t^{-\alpha}(1-t)^{-\tau}\tfrac{\D t}t\\[-15mm]
 &\quad=(1-s)^{-\alpha-\tau}\int_{c-i\infty}^{c+i\infty}t^{-\alpha}(1- t)^{-\tau}\tfrac{\D t}t
\ \ \ 
\if0
\raisebox{-4mm}{\makebox[0pt][l]{\scalebox{0.95}{\begin{tikzpicture}
% os_md.xyang(0.3,[0,0],@pi/2,3*@pi/2);
\draw (2.5,0.3)--(0,0.3) .. controls (-0.166,0.3) and (-0.3,0.166) .. (-0.3,0)
 .. controls (-0.3,-0.166) and (-0.166,-0.3) .. (0,-0.3)--(2.5,-0.3)
(0,0)--(2.5,0);
% os_md.xyang(0.2,[3,0.3],[0,0.3],2);
\draw (2.359,0.159) -- (2.5,0.3) -- (2.359,0.441);
\draw[dotted] (-2,0)--(0,0);
\draw (-0.5,-1)--(-0.5,1) (-0.359,0.859) -- (-0.5,1) -- (-0.641,0.859);
\node at (0,0) {$\bullet$};
\node at (0,-0.18) {\scriptsize $1$};
\node at (2.8,0) {\scriptsize$+\infty$};
\node at (-0.5,0) {$\bullet$};
\node at (-0.65,-0.18) {\scriptsize $c$};
\node at (-1,0) {$\bullet$};
\node at (-1,-0.18) {\scriptsize $0$};
\node at (-0.5,1.15) {\scriptsize $c{+}i\infty$};
\end{tikzpicture}}}}\\[-1mm]
\fi
\raisebox{-4mm}{\makebox[0pt][l]{\scalebox{0.95}{\begin{tikzpicture}
% os_md.xyang(0.3,[0,0],@pi/2,3*@pi/2);
\draw (2.5,0.3)--(0.5,0.3) .. controls (0.334,0.3) and (0.2,0.166) .. (0.2,0)
 .. controls (0.2,-0.166) and (0.344,-0.3) .. (0.5,-0.3)--(2.5,-0.3)
(0.5,0)--(2.5,0);
\draw (2.359,0.159) -- (2.5,0.3) -- (2.359,0.441);
\draw[dotted] (-2,0)--(0.5,0);
\draw (-0.5,-1)--(-0.5,1) (-0.359,0.859) -- (-0.5,1) -- (-0.641,0.859);
% Q=[0.5,0.9];P=[-1.3,-1];os_md.xyang(0.2,Q,P,2);
\draw (-1.1,-1)--(0.5,0.9) (0.517,0.701) -- (0.5,0.9) -- (0.301,0.883);
\node at (0.5,0) {$\bullet$};
\node at (0.5,-0.18) {\scriptsize $1$};
\node at (2.8,0) {\scriptsize$+\infty$};
\node at (-0.5,0) {$\bullet$};
\node at (-0.65,-0.18) {\scriptsize $c$};
\node at (-1,0) {$\bullet$};
\node at (-1,-0.18) {\scriptsize $0$};
\node at (-0.5,1.15) {\scriptsize $c{+}i\infty$};
\node at (0.9,1) {\scriptsize$\frac{c{+}i\infty}{1-s}$};
\end{tikzpicture}}}}\\[-1mm]
&\quad=(1-s)^{-\alpha-\tau}(-e^{-\tau\pi i}+e^{\tau\pi i})\int_1^\infty t^{-\alpha}(t-1)^{-\tau}\tfrac{\D t}t\\
&\quad=(1-s)^{-\alpha-\tau}\cdot 2 i \sin\tau\pi 
 \int_0^1 (\tfrac1u)^{-\alpha}(\tfrac1u-1)^{-\tau}\tfrac{\D u}u\quad(u=\tfrac1t)\\
&\quad=\frac{2\pi i(1-s)^{-\alpha-\tau}}{\Gamma(\tau)\Gamma(1-\tau)}
 \int_0^1 u^{\alpha+\tau-1}(1-u)^{-\tau}\D u\\
%&=2\pi i\frac{\Gamma(\alpha_1+\alpha_2-1)}{\Gamma(\alpha_1)\Gamma(\alpha_2)}
 &\quad=2\pi i\frac{\Gamma(\alpha+\tau)}{\Gamma(\tau)\Gamma(\alpha+1)}(1-s)^{-\alpha-\tau}.
\end{align*}
%%%
Here the path of the integration of the above first line is $(-\infty,\infty)\ni s\mapsto c+is$ and 
we also use the path \ \ \raisebox{-3mm}{\scalebox{0.9}{\begin{tikzpicture}
% os_md.xyang(0.3,[0,0],@pi/2,3*@pi/2);
\draw (2.5,0.3)--(0,0.3) .. controls (-0.166,0.3) and (-0.3,0.166) .. (-0.3,0)
 .. controls (-0.3,-0.166) and (-0.166,-0.3) .. (0,-0.3)--(2.5,-0.3)
(0,0)--(2.5,0);
% os_md.xyang(0.2,[3,0.3],[0,0.3],2);
\draw (2.359,0.159) -- (2.5,0.3) -- (2.359,0.441);
\node at (0,0) {$\bullet$};
\node at (0,-0.18) {\scriptsize $1$};
\node at (2.8,0) {\scriptsize$+\infty$};
\end{tikzpicture}}}\ \ 
of the integration in the above.
%%%
\if0
\begin{align*}
\frac{\Gamma(\mu+1)}{2\pi i}\int_{c-\infty i}^{c+\infty i}
  (\tfrac xt)^{\alpha}(1-t)^{-\mu-1}\tfrac{\D t}t
&=\frac{\Gamma(\alpha+\mu+1)}{\Gamma(\alpha+1)}x^\alpha
\end{align*}
\fi
%%%

Thus we have 
\begin{align*}
&\int_{\tfrac1{n+1}-i\infty}^{\tfrac1{n+1}+i\infty}\!\cdots\!
\int_{\tfrac1{n+1}-i\infty}^{\tfrac1{n+1}+i\infty}
t^{-\boldsymbol\alpha}(1-t_1-\cdots-t_n)^{-\tau}\tfrac{\D t_1}{t_1}\cdots\tfrac{\D t_n}{t_n}\\
&=(2\pi i)^n\frac{\Gamma(\alpha_n+\tau)}{\Gamma(\tau)\Gamma(\alpha_n+1)}
    \frac{\Gamma(\alpha_{n-1}+\alpha_n+\tau)}{\Gamma(\alpha_n+\tau)\Gamma(\alpha_{n-1}+1)}
\cdots \frac{\Gamma(|\boldsymbol\alpha|+\tau)}
  {\Gamma(\alpha_2+\cdots+\alpha_n+\tau)\Gamma(\alpha_1+1)}\\
&=(2\pi i)^n\
 \frac{\Gamma(|\boldsymbol\alpha+1|+\tau-n)}{\Gamma(\boldsymbol\alpha+1)\Gamma(\tau)}.
\end{align*}
\begin{definition}
We define the transformation 
\begin{equation*}
\begin{split}
 (L_x^\mu \phi)(x)&:=\frac{\Gamma(\mu+n)}{(2\pi i)^n}\int_{\tfrac1{n+1}-i\infty}^{\tfrac1{n+1}+i\infty}\!\cdots\!
\int_{\tfrac1{n+1}-i\infty}^{\tfrac1{n+1}+i\infty}\!
\phi(\tfrac{x_1}{t_1},\dots,\tfrac{x_n}{t_n})
(1-|\mathbf t|)^{-\mu-n}\tfrac{\D t_1}{t_1}\cdots\tfrac{\D t_n}{t_n}\\
 &=\frac{\Gamma(\mu+n)}{(2\pi i)^n}\int_{\bigl|\tfrac{2s_1}{n+1}-1\bigr|=1}\!\cdots\!
   \int_{\bigl|\tfrac{2s_n}{n+1}-1\bigr|=1}\!\phi(s_1x_1,\dots,s_nx_n)\\
 &\qquad\times
  (1-\tfrac1{s_1}-\cdots-\tfrac1{s_n})^{-\mu-n}\tfrac{\D s_1}{s_1}\cdots\tfrac{\D s_n}{s_n}.
\end{split}
\end{equation*}
\end{definition}
Then we have 
\begin{equation}\label{eq:Lx0}
  L_x^\mu x^{\boldsymbol\alpha}=
 \frac{\Gamma(|\boldsymbol\alpha+1|+\mu)}{\Gamma(\boldsymbol\alpha+1 )}
  x^{\boldsymbol\alpha}.
\end{equation}

When $n=1$, we have
\begin{align*}
 (L_x^{\mu}\phi)(x)&=\frac{\Gamma(\mu+1)}{2\pi i}\int_{\tfrac12-i\infty}^{\tfrac12+i\infty}
  \phi(\tfrac xt)(1-t)^{-\mu-1}\tfrac{\D t}t\\
  &=\frac{\sin(\mu+1)\pi\cdot\Gamma(\mu+1)}{\pi}
   \int_1^\infty\phi(\tfrac xt)(1-t)^{-\mu-1}\tfrac{\D t}t\\
  &=\frac1{\Gamma(-\mu)}\int_0^1 s^{\mu}\phi(s)(x-s)^{-\mu-1}\D s\\
  &=I_0^{-\mu} (x^\mu\phi).
\end{align*}
In general, we have
\begin{equation}\label{eq:Ls}
\begin{split}
 &L_x^{\mu}(\phi(x_1)x_2^{\alpha_2-1}\cdots x_n^{\alpha_n-1})=\frac{\Gamma(\alpha_2+\cdots+\alpha_n+\mu)}{\Gamma(\alpha_2)\cdots\Gamma(\alpha_n)}\\
 &\quad\times (I_{x_1,0}^{-\alpha_2-\cdots-\alpha_n-\mu}x_1^{\alpha_2+\cdots+\alpha_n+\mu}\phi(x_1)) x_2^{\alpha_2-1}\cdots x_n^{\alpha_n-1}.
\end{split}
\end{equation}
\begin{definition}
We define two transformations
\begin{align}
 K_x^{\mu,\boldsymbol\lambda}:=x^{1-\boldsymbol\lambda}K_x^\mu x^{\boldsymbol\lambda-1}
  \text{ \ and \ }
 L_x^{\mu,\boldsymbol\lambda}:=x^{1-\boldsymbol\lambda}L_x^\mu x^{\boldsymbol\lambda-1}
\end{align}
which act on the ring $\mathcal O_0$ of convergent power series of $x=(x_1,\dots,x_n)$.
\end{definition}
We have
\begin{align*}
 K_x^{\mu,\boldsymbol\lambda}x^{\boldsymbol\alpha}
  =\frac{\Gamma(\boldsymbol\lambda+\boldsymbol\alpha)}{\Gamma(|\boldsymbol\lambda+\boldsymbol\alpha|+\mu)}x^{\boldsymbol\alpha}\text{ \ and \ }
 L_x^{\mu,\boldsymbol\lambda}x^{\boldsymbol\alpha}
  =\frac{\Gamma(|\boldsymbol\lambda+\boldsymbol\alpha|+\mu)}{\Gamma(\boldsymbol\lambda+\boldsymbol\alpha)}x^{\boldsymbol\alpha}.
\end{align*}
\begin{theorem}\label{thm:KL}
Putting
\begin{align*}
  u(x)&
  =\sum_{\mathbf m\ge0}c_{\mathbf m}x^{\mathbf m}=
   \sum_{m_1=0}^\infty\cdots \sum_{m_n=0}^\infty c_{\mathbf m}x^{\mathbf m}
   \in\mathcal O_0\quad(c_{\mathbf m}\in\mathbb C).\\
  K_x^{\mu,\boldsymbol\lambda}u(x)&=\sum_{\mathbf m}c_{\mathbf m}^K x^{\mathbf m}\text{ \ and \ } 
  L_x^{\mu,\boldsymbol\lambda}u(x)=\sum_{\mathbf m}c_{\mathbf m}^L x^{\mathbf m}\quad
 (c^K_{\mathbf m},\ c^L_{\mathbf m}\in\mathbb C), 
\end{align*} 
we have
\begin{align}
 c_{\mathbf m}^K
  &=\frac{\Gamma(\boldsymbol\lambda)}{\Gamma(|\boldsymbol\lambda|+\mu)}
  \frac{(\boldsymbol\lambda)_{\mathbf m}}{(|\boldsymbol\lambda|+\mu)_{|\mathbf m|}}c_{\mathbf m},\\
 c_{\mathbf m}^L
  &=\frac{\Gamma(|\boldsymbol\lambda|+\mu)}{\Gamma(\boldsymbol\lambda)}
  \frac{(|\boldsymbol\lambda|+\mu)_{|\mathbf m|}}{(\boldsymbol\lambda)_{\mathbf m}}c_{\mathbf m}. 
\end{align}
\end{theorem}
By the analytic continuation with respect to the parameters the transformations 
$K_x^{\mu,\boldsymbol \lambda}$ and $L_x^{\mu,\boldsymbol \lambda}$ are well-defined if
\begin{align}
 \lambda_\nu&\notin \mathbb Z_{\le 0}\quad(\nu=1,\dots,n)\label{eq:Kd}
\intertext{and}
 |\boldsymbol\lambda|+\mu&\notin \mathbb Z_{\le 0},\label{eq:Ld}
\end{align}
respectively.
Moreover if \eqref{eq:Kd} and \eqref{eq:Ld} are valid, $K_x^{\mu,\boldsymbol \lambda}$ and $L_x^{\mu,\boldsymbol \lambda}$ are bijective on $\mathcal O_0$ and the map $K_x^{\mu,\boldsymbol \lambda}\circ L_x^{\mu,\boldsymbol \lambda}$ is the identity map.

\begin{proposition}\label{prop:to1}
By the equations \eqref{eq:Ks} and \eqref{eq:Ls} we have
\begin{align}
 K_x^{\mu,\boldsymbol\lambda}\phi(x_1)&=
  \frac{\Gamma(\lambda_2)\cdots\Gamma(\lambda_n)}{\Gamma(\lambda_2+\cdots+\lambda_n+\mu)}
  x_1^{-|\boldsymbol\lambda|-\mu+1}
  I_{x_1,0}^{|\boldsymbol\lambda|-\lambda_1+\mu}x_1^{\lambda_1-1}\phi(x_1),\\
  L_x^{\mu,\boldsymbol\lambda}\phi(x_1)&=
  \frac{\Gamma(\lambda_2+\cdots+\lambda_n+\mu)}{\Gamma(\lambda_2)\cdots\Gamma(\lambda_n)}
  x_1^{1-\lambda_1}I_{x_1,0}^{-|\boldsymbol\lambda|+\lambda_1-\mu}
  x_1^{|\boldsymbol\lambda|+\mu-1}\phi(x_1). 
\end{align}
\end{proposition}
%%%%%%%%%%%%%%%%%%%%%%%%%%%%%%%%%%%%%%%%%%%%%%%%%%%%%%%%%%%
\section{Some hypergeometric functions}
Under the notation \eqref{notation} we note that 
\[
 (1-|\mathbf x|)^{-\lambda}=\sum_{\mathbf m\ge0}\frac{(\lambda)_{|\mathbf m|}}
 {\mathbf m!}x^{\mathbf m}\text{ \ and \ }e^{|\mathbf x|}=\sum_{\mathbf m\ge0}\frac{\mathbf x^{\mathbf m}}{\mathbf m!}.
\]
Lauricella's hypergeometric series (cf.~\cite{La,Er}) 
and their integral representations are given as follows
(cf.~Theorem~\ref{thm:KL}).
\begin{align}
\begin{split}
 F_A(\lambda_0,\boldsymbol \mu,\boldsymbol\lambda;\mathbf x)&:=
 \sum_{\mathbf m\ge0}\frac{(\lambda_0)_{|\mathbf m|}(\boldsymbol\mu)_{\mathbf m}}
  {(\boldsymbol\lambda)_{\mathbf m}\mathbf m!}\mathbf x^{\mathbf m}\\
 &=\frac{\Gamma(\boldsymbol\lambda)}{\Gamma(\boldsymbol\mu)}
  K_{x_1}^{\lambda_1-\mu_1,\mu_1}\cdots 
  K_{x_n}^{\lambda_n-\mu_n,\mu_n}
  (1-|\mathbf x|)^{-\lambda_0}\\
 &=\frac{\Gamma(\boldsymbol\lambda)}{\Gamma(\lambda_0)}
   L_x^{\lambda_0-|\boldsymbol\lambda|,\boldsymbol\lambda}(1-\mathbf x)^{-\boldsymbol\mu}
\end{split}
\label{eq:FA}
\allowdisplaybreaks\\
%%%
F_B(\boldsymbol\lambda,\boldsymbol\lambda',\mu;\mathbf x)&:=
 \sum_{\mathbf m\ge0}\frac{(\boldsymbol\lambda)_{\mathbf m}(\boldsymbol\lambda')_{\mathbf m}}
 {(\mu)_{|\mathbf m|}\mathbf m!}{\mathbf x}^m
 =\frac{\Gamma(\mu)}{\Gamma(\boldsymbol\lambda)}
 K_x^{\mu-|\boldsymbol \lambda|,\boldsymbol\lambda}(1-\mathbf x)^{-\boldsymbol\lambda'},
\label{eq:FB}
\allowdisplaybreaks
\\
F_C(\mu,\lambda_0,\boldsymbol\lambda;\mathbf x)
 &:=\sum_{\mathbf m\ge0} \frac{(\mu)_{|\mathbf m|}(\lambda_0)_{|\mathbf m|}}
  {(\boldsymbol\lambda)_{\mathbf m}\mathbf m!}{\mathbf x}^{\mathbf m}
  =\frac{\Gamma(\boldsymbol\lambda)}{\Gamma(\mu)}
   L_x^{\mu-|\boldsymbol\lambda|,\boldsymbol\lambda}(1-|\mathbf x|)^{-\lambda_0},
\label{eq:FC}
\allowdisplaybreaks\\
F_D(\lambda_0,\boldsymbol\lambda,\mu;\mathbf x)&:=
 \sum_{\mathbf m\ge0}\frac{(\lambda_0)_{|\mathbf m|}(\boldsymbol\lambda)_{\mathbf m}}
 {(\mu)_{|\mathbf m|}{\mathbf m}!}
  {\mathbf x}^{\mathbf m}
 =\frac{\Gamma(\mu)}{\Gamma(\boldsymbol\lambda)}
K_{x}^{\mu-|\boldsymbol \lambda|,\boldsymbol\lambda}(1-|\mathbf x|)^{-\lambda_0}.\label{eq:FD}
\end{align}

When $n=2$, namely, the number of variables equals 2, the functions $F_D$, $F_A$, $F_B$ and $F_C$ are Appell's hypergeometric series $F_1$, $F_2$, $F_3$ and $F_4$ (cf.~\cite{AK}), respectively.
Moreover we give examples of confluent Horn's series (cf.~\cite{Ho,Er}): 
\begin{align}
\begin{split}
 \Phi_2(\beta,\beta';\gamma;x,y)&:=\sum_{m=0}^\infty\sum_{n=0}^\infty
 \frac{(\beta)_m(\beta')_n}{(\gamma)_{m+n}m!n!}x^my^n
 \\&=
\frac{\Gamma(\gamma)}{\Gamma(\beta)\Gamma(\beta')}
 K_{x,y}^{\gamma-\beta-\beta',\beta,\beta'}e^{x+y},
\end{split}
\allowdisplaybreaks\\
\begin{split}
 \Psi_1(\alpha;\beta;\gamma,\gamma';x,y)&:=\sum_{m=0}^\infty\sum_{n=0}^\infty
 \frac{(\alpha)_{m+n}(\beta)_m}{(\gamma)_m(\gamma')_n m!n!}x^my^n
 \\&=
 \frac{\Gamma(\gamma)\Gamma(\gamma')}{\Gamma(\alpha)}
 L_{x,y}^{\alpha-\gamma-\gamma',\gamma,\gamma'}(1-x)^{-\beta}e^y,
\end{split}
\allowdisplaybreaks\\
\begin{split}
 \Psi_2(\alpha;\gamma',\gamma';x,y)&:=\sum_{m=0}^\infty\sum_{n=0}^\infty
 \frac{(\alpha)_{m+n}}{(\gamma)_m(\gamma')_n m!n!}x^my^n
 \\&=
 \frac{\Gamma(\gamma)\Gamma(\gamma')}{\Gamma(\alpha)}
 L_{x,y}^{\alpha-\gamma-\gamma',\gamma,\gamma'}e^{x+y}.
\end{split}
\end{align}

When $n=1$, \eqref{eq:FA}, \eqref{eq:FB}, \eqref{eq:FC} and \eqref{eq:FD} 
are reduced to an integral representation of Gauss hypergeometric series. 
In fact, putting $(\lambda_0,\boldsymbol\mu,\boldsymbol\lambda)=(\alpha,\beta,\gamma)$ in 
\eqref{eq:FA}, we have
\begin{equation}\label{eq:Gauss}
\begin{split}
 F(\alpha,\beta,\gamma;x)&=\sum_{m=0}^\infty\frac{(\alpha)_m(\beta)_m}{(\gamma)_mm!}x^m\\
 &=\frac{\Gamma(\gamma)}{\Gamma(\alpha)}K_x^{\gamma-\alpha,\alpha}(1-x)^{-\beta}\\
 &=
   \frac{\Gamma(\gamma)}{\Gamma(\alpha)\Gamma(\gamma-\alpha)}
   \int_0^1t^{\alpha-1}(1-t)^{\gamma-\alpha-1}(1-tx)^{-\beta}\D t
\end{split}
\end{equation}
and Kummer function is
\begin{align}
\begin{split}
 {}_1F_1(\alpha;\gamma;x)&:=\sum_{n=0}^\infty \frac{(\alpha)_n}{(\gamma)_n}\frac{x^n}{n!}\\
 &=\frac{\Gamma(\gamma)}{\Gamma(\alpha)}K_x^{\gamma-\alpha,\alpha}e^x
 =\frac{\Gamma(\gamma)}{\Gamma(\alpha)\Gamma(\gamma-\alpha)}\int_0^1t^{\alpha-1}(1-t)^{\gamma-\alpha-1}e^{tx}\D t.
\end{split}
\end{align}
\section{A connection problem}
Integral representations of hypergeometric functions are useful for the study of 
global structure of the functions.  Rigid linear ordinary differential equations
on the Riemann sphere with regular or unramified irregular irregularities are reduced
to the trivial equation by successive applications of middle convolutions and additions.
These transformations correspond to the transformations 
of their solutions defined by Riemann-Liouville 
integrals and multiplications by elementary functions such as $(x-c)^\lambda$ or 
$e^{r(x)}$ with rational functions $r(x)$ of $x$.

In \cite[Chapter 12]{Ow} and \cite{Ori}, analyzing the asymptotic behavior of the Riemann Liouville integral when the variable $x$ tends to a singular point of the function,  
we get the change of connection coefficients and 
Stokes coefficients under the integral transformations and 
finally such coefficients of the hypergeometric function we are interested in.

In this section a generalization of this way of study is shown in the case of several variables,
which will be explained by using $F_1$.
Since
\begin{align*}
 F_1(a,b,b',c;x,y)&=\frac{\Gamma(c)}{\Gamma(b)\Gamma(b')}K_{x,y}^{c-b-b',b,b'}(1-x-y)^{-a},
\intertext{Proposition~\ref{prop:to1} implies}
 F_1(a,b,b',c;x,0)&=\frac{\Gamma(c)}{\Gamma(b)}x^{1-c}I_0^{c-b}
  x^{b-1}(1-x)^{-a}.
\end{align*}
The equalities \eqref{eq:K2I} and \eqref{eq:Gauss} show  $F_1(a,b,b',c;x,0)=F(b,a,c;x)$
but first we do not use this fact. 

We pursue the changes of the Riemann scheme under the procedure given by 
Proposition~\ref{prop:to1} (cf.~\cite[Chapter~5]{Ow}).
They are the change when  
we apply $I_0^{c-b}$ to $x^{b-1}(1-x)^{-a}$ and the change when we  
multiply the resulting function by $\frac{\Gamma(c)}{\Gamma(b)}x^{1-c}$, which are
\begin{align}
&x^{b-1}(1-x)^{-a}:
\begin{Bmatrix}
 x=0&1&\infty\\
 \underline{b-1}&-a&a-b+1
\end{Bmatrix}\notag\\
&\qquad
\xrightarrow{I_0^{c-b}}
\begin{Bmatrix}
 x=0&1&\infty\\
 0 & 0 & b-c+1\\
 \underline{c-1}&c-a-b&a-c+1
\end{Bmatrix}\notag\\
&\qquad
\xrightarrow{\times\frac{\Gamma(c)}{\Gamma(b)}x^{1-c}}
\begin{Bmatrix}
 x=0&1&\infty\\
 1-c & 0 & b\\
 \underline{0}&c-a-b&a
\end{Bmatrix}.\label{eq:RSG}
\end{align}
Then  $F_1(a,b,b',c;x,0)$ is characterized as the holomorphic function in a neighborhood of 0
with the Riemann scheme \eqref{eq:RSG} and $F_1(a,b,b',c;0,0)=1$.

Since $F_1(a,b,b',c;x,y)$ satisfies a system of differential equations with singularities
$x=0,1,\infty$, $y=0,1,\infty$ and $x=y$ in $\mathbb P^1\times\mathbb P^1$, we have a connection
relation
\begin{align}\label{eq:F1c}
 F_1(a,b,b',c;x,y)=(-\tfrac1x)^a C_af_1^a(x,y)
    +(-\tfrac1x)^bC_bf_1^b(x,y)
\end{align}
in a neighborhood of $(-\infty,0)\times \{0\}$ in $\mathbb P^1\times\mathbb P^1$.
Here $f_1^a(x,y)$ and $f_1^b(x,y)$ are holomorphic in a neighborhood of $[-\infty,0)\times\{0\}$ and
$f_1^a(\infty,0)=f_1^b(\infty,0)=1$ and the connection coefficients $C_a$ and $C_b$ are
given by those of $F_1(a,b,b',c;x,0)$, namely,
\begin{align}
 C_a=\frac{\Gamma(c)\Gamma(b-a)}{\Gamma(b)\Gamma(c-a)}\text{\ \ and \ \ }
 C_b=\frac{\Gamma(c)\Gamma(a-b)}{\Gamma(a)\Gamma(c-b)}.
\end{align}
\begin{remark}
We note that the general expression  \cite[(0.25)]{Ow} of the connection coefficient
of Gauss hypergeometric function is simple and easy to be specialized (cf.~\cite{Ogauss}). 
Moreover the connection formula
\begin{align*}
 F(a,b,c;x)&=(-\tfrac1x)^aC_a F(a,a-c+1,a-b+1;\tfrac1x)\\
  &\quad{}+(-\tfrac1x)^bC_b F(b,b-c+1,b-a+1;\tfrac1x)
\end{align*}
follows from
{\small\begin{align*}
 \begin{Bmatrix}
 x=0 & 1 & \infty\\
 0 & 0 & a;&x\\
 1-c&c-a-b&b
 \end{Bmatrix}&= \begin{Bmatrix}
 x=0 & 1 & \infty\\
 a& 0 & 0;&\tfrac1x\\
 b& c-a-b & 1-c
 \end{Bmatrix}\\
 &=(-x)^a
 \begin{Bmatrix}
 x=0 & 1 & \infty\\
 0& 0 & a;&\tfrac1x\\
 1-(a-b+1)& c-a-b & a-c+1
 \end{Bmatrix}.
\end{align*}}
\end{remark}

We explicitly calculate \eqref{eq:F1c} in the following way.
\begin{align*}
&F_1(a,b,b',c;x,y)=\sum_{m=0}^\infty\sum_{n=0}^\infty
  \frac{(a)_{m+n}(b)_m(b')_n}{(c)_{m+n}m!n!}x^my^n
\allowdisplaybreaks\\
&\quad=\sum_{n=0}^\infty \frac{(a)_n(b')_n}{(c)_nn!}y^n
  \sum_{m=0}^\infty \frac{(a+n)_m(b)_m}{(c+n)_mm!}x^m
 =\sum_{n=0}^\infty \frac{(a)_n(b')_n}{(c)_nn!}y^nF(a+n,b,c+n;x)
\allowdisplaybreaks\\
&\quad=\sum_{n=0}^\infty \frac{(a)_n(b')_n}{(c)_nn!} \Bigl(\\
&\qquad\quad (-x)^{-a-n}
  \frac{\Gamma(c+n)\Gamma(b-a-n)}{\Gamma(b)\Gamma(c-a)} F(a+n,a-c+1,a-b+n+1;\tfrac1x)\\
 &\qquad{} +(-x)^{-b}
 \frac{\Gamma(c+n)\Gamma(a-b+n)}{\Gamma(a+n)\Gamma(c-b+n)}
  F(b,b-c-n+1,b-a-n+1;\tfrac1x)\Bigr)\\
&\quad=(-\tfrac1x)^aF_1^a+(-\tfrac1x)^bF_1^b,\\
%%%%%%
&F_1^a=\sum\frac{\Gamma(c)(a)_n(b')_n\Gamma(b-a-n)(a+n)_m(a-c+1)_m}
 {\Gamma(b)\Gamma(c-a)(c)_n(a-b+n+1)_mm!n!}(\tfrac1x)^m(-\tfrac yx)^n
\allowdisplaybreaks\\
&\quad =
 \sum\frac{\Gamma(c)(a)_{m+n}(b')_n\Gamma(a-b+1)\Gamma(b-a)(a-c+1)_m}
 {\Gamma(b)\Gamma(c-a)\Gamma(a-b+n+1)(a-b+n+1)_mm!n!}(\tfrac1x)^m(\tfrac yx)^n
\allowdisplaybreaks\\
&\quad=
 \sum\frac{\Gamma(c)\Gamma(a-b+1)\Gamma(b-a)(a)_{m+n}(b')_n(a-c+1)_m}
 {\Gamma(b)\Gamma(c-a)\Gamma(a-b+1+m+n)m!n!}(\tfrac1x)^m(\tfrac yx)^n
\allowdisplaybreaks\\
&\quad=
 \sum\frac{\Gamma(c)\Gamma(b-a)(a)_{m+n}(b')_n(a-c+1)_m}
 {\Gamma(b)\Gamma(c-a)(a-b+1)_{m+n}m!n!}(\tfrac1x)^m(\tfrac yx)^n
\allowdisplaybreaks\\
&\quad=\sum\frac{\Gamma(c)\Gamma(b-a)(a)_{m+n}(a-c+1)_m(b')_n}
 {\Gamma(b)\Gamma(c-a)(a-b+1)_{m+n}m!n!}(\tfrac1x)^m(\tfrac yx)^n\allowdisplaybreaks\\
&\quad=\frac{\Gamma(c)\Gamma(b-a)}{\Gamma(b)\Gamma(c-a)}
 F_1(a,a-c+1,b',a-b+1;\tfrac1x,\tfrac yx),\allowdisplaybreaks\\
%%%
&F_1^b=\sum \frac{\Gamma(c)(b')_n\Gamma(a-b+n)(b)_m(b-c-n+1)_m}
 {\Gamma(a)\Gamma(c-b+n)(b-a-n+1)_mm!n!}(\tfrac1x)^my^n
\allowdisplaybreaks\\
&\quad=\sum \frac{\Gamma(c)(b)_m(b')_n\Gamma(a-b+n)\Gamma(b-c-n+1)
   (b-c-n+1)_m(-\tfrac1x)^m(-y)^n}
 {\Gamma(a)\Gamma(c-b)\Gamma(b-c+1)(a-b+n-1)\cdots(a-b+n-m)m!n!}
\allowdisplaybreaks\\
&\quad=\sum \frac{\Gamma(c)(b)_m(b')_n\Gamma(a-b+n-m)\Gamma(b-c-n+1+m)
  (-\tfrac1x)^m(-y)^n}
 {\Gamma(a)\Gamma(c-b)\Gamma(b-c+1)m!n!}
\allowdisplaybreaks\\
&\quad=\sum 
 \frac{\Gamma(c)(b)_m(b')_n\Gamma(a-b)(a-b)_{n-m}(b-c+1)_{m-n}}
 {\Gamma(a)\Gamma(c-b)m!n!}(-\tfrac1x)^m(-y)^n
\allowdisplaybreaks\\
&\quad=\frac{\Gamma(c)\Gamma(a-b)}{\Gamma(a)\Gamma(c-b)}
 G_2(b,b',a-b,b-c+1;-\tfrac1x,-y).
\end{align*}
Here
\begin{equation}
 G_2(\alpha,\beta,\gamma,\delta;x,y):=\sum_{m=0}^\infty\sum_{n=0}^\infty
 \frac{(\alpha)_m(\beta)_n(\gamma)_{n-m}(\delta)_{m-n}}{m!n!}x^my^n
\end{equation}
and we have
\begin{equation}
\begin{split}
f_1^a(x,y)&=F_1(a,a-c+1,b',a-b+1;\tfrac1x,\tfrac yx),\\
f_1^b(x,y)&=G_2(b,b',a-b,b-c+1;-\tfrac1x,-y)
\end{split}
\end{equation}
in \eqref{eq:F1c}.

\begin{remark}
The argument above is justified since $F_1(a;b,b';c;x,y)$ satisfies a differential equation 
which has regular singularities along the hypersurface
$(-\infty,0)\times\{0\}$ of $\mathbb C^2$ (cf.~\cite{KO}) or Kummer's formula
\[
  F(\alpha,\beta,\gamma;x)=(1-x)^{-\alpha}F(\alpha,\gamma-\beta,\gamma;\tfrac x{x-1}).
\]
\end{remark}

We give an answer to a part of connection problem of 
Appell's $F_1$ which satisfies a KZ equation of rank 3 and
the equation allows the coordinate transformations on $(\mathbb P^1)^5$ corresponding 
to the permutations of 5 coordinates.
By the action of this transformation we get Kummer type formula for $F_1$ and solve 
the connection problem for $F_1$.  Note that the singularity of the origin is not of
the normally crossing type but for example,
the map $(x,y)\mapsto (\frac1x,\frac yx)$ is one of the coordinate transformations
and the blowing up of the origin naturally corresponds to this coordinate 
transformation.  This enables us to get all the analytic continuation of $F_1$ 
in $\mathbb P^1\times \mathbb P^1$  in terms of  $F_1$ and $G_2$ as is 
given in the above special case (cf.~\S\ref{sec:KZ}).
Another independent solution $f_1^c$ of the equation at $(\infty,0)$ is characterized by 
the fact that the analytic continuation of $f_1^c$ in a suitable neighborhood of 
$[\infty,0)\times \{0\}$ is a scalar multiple of $f_1^c$ and then 
$f_1^c$ is expressed by using $F_1$
as in the case of $(-\tfrac1x)^af_1^a$.
This will be discussed in another paper with more general examples.
%%%%%%%%%%%%%%%%%%%%%%%%% More trans. %%%%%%%%%%%%%%%%%%%%%%%%%%%%%%%%%%%%%%%%%%
\section{More transformations}
In this section we examine transformations of power series obtained by a suitable class of 
coordinate transformations and the transformations $K_x^{\mu,\boldmath \lambda}$ and 
$L_x^{\mu,\boldmath\lambda}$.
For a coordinate transformation $x\mapsto R(x)$ of $\mathbb C^n$ we put
\[
   (T_{x\to R(x)}\phi)=\phi(R(x))
\]
for functions $\phi(x)$.
\begin{definition}
Choose a subset of indices $\{i_1,\dots,i_k\}\subset \{1,\dots,n\}$ 
and put $\mathbf y=(x_{i_1},\dots,x_{i_k})$.
Then for $\mu\in\mathbb C$ and $\boldsymbol\lambda\in\mathbb C^k$ we define
\begin{align*}
K_{\mathbf y,x\to R(x)}^{\mu,\boldsymbol\lambda}
 &:=T_{x\to R(x)}^{-1}\circ
 K_{\mathbf y}^{\mu,\boldsymbol \lambda}\circ T_{x\to R(x)},\\
L_{\mathbf y,x\to R(x)}^{\mu,\boldsymbol\lambda}
 &:=T_{x\to R(x)}^{-1}\circ
 L_{\mathbf y}^{\mu,\boldsymbol\lambda}\circ T_{x\to R(x)}.
\end{align*}
\end{definition}
%%%
Let $\mathbf p=\Bigl(p_{i,j}\Bigr)_{\substack{1\le i\le n\\1\le j\le n}}\in GL(n,\mathbb Z)$.
We denote
\begin{equation}
\begin{split}
 x^{\mathbf p}&={\mathbf x}^{\mathbf p}=(x^{p_{*,1}},\dots,x^{p_{*,n}})
       =\Bigl(\prod_{\nu=1}^nx_\nu^{p_{\nu,1}},\dots,\prod_{\nu=1}^nx_\nu^{p_{\nu,n}}\Bigr),\\
 \mathbf p\mathbf m&=(p_{1,*}\mathbf m,\dots,p_{n,*}\mathbf m)
       =\Bigl(\sum_{\nu=1}^np_{1,\nu}m_\nu,\dots,\sum_{\nu=1}^n{p_{n,\nu}}m_\nu\Bigr)
\end{split}
\end{equation}
with $\mathbf m=(m_1,\dots,m_n)\in\mathbb Z^n$.  Then 
$T_{x\to x^{\mathbf p}}^{-1}=T_{x\to x^{\mathbf p^{-1}}}$

We examine the transformations 
$K_{\mathbf y,x\to x^{\mathbf p}}^{\mu,\boldsymbol\lambda}$ and 
$L_{\mathbf y,x\to x^{\mathbf p}}^{\mu,\boldsymbol\lambda}$
under the assumption
\begin{equation}\label{eq:positive}
 p_{i_\nu,j}\ge 0\qquad(1\le \nu\le k,\ 1\le j\le n).
\end{equation}
Since
\[
 \bigl(T_{x\to x^{\mathbf p}}^{-1}T_{x\to(t_1x_1,\dots,t_nx_n)}T_{x\to x^{\mathbf p}}\phi\bigr)(x)=
 \phi\Bigl(x_1\prod_{\nu=1}^n t_\nu^{p_{\nu,1}},\dots,x_n\prod_{\nu=1}^n t_\nu^{p_{\nu,n}}\Bigr),
\]
we have
\begin{align}\label{eq:Kdef}
\begin{split}
&\bigl(K_{(x_{i_1},\dots,x_{i_k}),x\to x^{\mathbf p}}^{\mu,\boldsymbol\lambda}\phi\bigr)(x)\\
&\quad=\frac1{\Gamma(\mu)}\int_{\substack{t_1>0,\dots t_k>0\\ t_1+\cdots+t_k<1}}
 \mathbf t^{\boldsymbol \lambda-1}(1-|\mathbf t|)^{\mu-1}
 \phi\Bigl(x_1\prod_{\nu=1}^k t_\nu^{p_{i_\nu,1}},\dots,x_n\prod_{\nu=1}^k t_\nu^{p_{i_\nu,n}}\Bigr)\D t,
\end{split}\\
\label{eq:Ldef}
\begin{split}
&\bigl(L_{(x_{i_1},\dots,x_{i_k}),x\to x^{\mathbf p}}^{\mu,\boldsymbol\lambda}\phi\bigr)(x)
=\frac{\Gamma(\mu+k)}{(2\pi i)^k}\int_{c-i\infty}^{c+i\infty}\cdots \int_{c-i\infty}^{c+i\infty}
 \mathbf t^{\boldsymbol \lambda-1}(1-|\mathbf t|)^{-\mu-k}\\
&\qquad
 \times\phi\Bigl(\frac{x_1}{\prod_{\nu=1}^k t_\nu^{p_{i_\nu,1}}},\dots,
 \frac{x_n}{\prod_{\nu=1}^k t_\nu^{p_{i_\nu,n}}}\Bigr)\frac{\D t_1}{t_1}\cdots\frac{\D t_k}{t_k}
 \quad\text{with \ }c=\tfrac1{k+1}.
\end{split}
\end{align}
We note that \eqref{eq:positive} assures that these transformations are defined on 
$\mathcal O_0$.  
\begin{proposition}
Denoting
\[
 (\mathbf p\mathbf m)_{i_1,\dots,i_k}
  =\Bigl(\sum_{\nu=1}^np_{i_1,\nu}m_\nu,\dots,\sum_{\nu=1}^n{p_{i_k,\nu}}m_\nu\Bigr),
\]
we have 
\begin{align}
K_{(x_{i_1},\dots,x_{i_k}),x\to x^{\mathbf p}}^{\mu,\boldsymbol\lambda}{x^\mathbf m}
=\frac{\Gamma(\boldsymbol\lambda + (\mathbf p\mathbf m)_{i_1,\dots,i_k})}
 {\Gamma(|\boldsymbol\lambda+(\mathbf p\mathbf m)_{i_1,\dots,i_k}|+\mu)}x^{\mathbf m},
\label{eq:Kco}\\
L_{(x_{i_1},\dots,x_{i_k}),x\to x^{\mathbf p}}^{\mu,\boldsymbol\lambda}{x^\mathbf m}
 =\frac {\Gamma(|\boldsymbol\lambda+(\mathbf p\mathbf m)_{i_1,\dots,i_k}|+\mu)}
 {\Gamma(\boldsymbol\lambda + (\mathbf p\mathbf m)_{i_1,\dots,i_k})}
 x^{\mathbf m}.\label{eq:Lco}
\end{align}
\end{proposition}
We give some examples hereafter in this section.

%\begin{example}
Let $(p_1,\dots,p_n)$ be a non-zero vector of non-negative integers.
Suppose the greatest common divisor of $p_1,\dots,p_n$ equals 1.
Then there exists $\mathbf p=\bigl(p_{i,j}\bigr)\in GL(n,\mathbb Z)$ with $p_{1,j}=p_j$
and
\[
K_{x_1,x\to \mathbf x^{\mathbf p}}^{\mu,\lambda}x^{\mathbf m}
 =\frac{\Gamma(\lambda)}{\Gamma(\lambda+\mu)}
 \frac{(\lambda)_{p_1m_1+\dots+p_nm_n}}{(\lambda+\mu)_{p_1m_1+\dots+p_nm_n}}x^{\mathbf m}.
\]

In particular we have
\begin{align}
K_{x_1,x\to (x_1,\tfrac{x_1}{x_2},\dots,\tfrac{x_1}{x_n})}^{\mu,\lambda}x^{\mathbf m}
&=\frac{\Gamma(\lambda)}{\Gamma(\lambda+\mu)}
 \frac{(\lambda)_{m_1+\dots+m_n}}{(\lambda+\mu)_{m_1+\dots+m_n}}x^\mathbf m,\\
F_D(\lambda_0,\boldsymbol\lambda,\mu;\mathbf x)
&=
 \frac{\Gamma(\mu)}{\Gamma(\lambda_0)}
 K_{x_1,x\to (x_1,\tfrac{x_1}{x_2},\dots,\tfrac{x_1}{x_n})}^{\mu-\lambda_0,\lambda_0}
 (1-\mathbf x)^{-\boldsymbol\lambda}\quad\text{(cf.~\eqref{eq:FD})}.\label{eq:FD2}
\end{align}
Here we note that the coordinate transformation 
$x\mapsto (x_1,\tfrac{x_1}{x_2},\dots,\tfrac{x_1}{x_n})$ gives a transformation of 
KZ equations of $n$ variables (cf.~\cite[\S6]{Okz}).

Let $\mathbf p=\begin{pmatrix}p_1&p_2\\q_1&q_2\end{pmatrix}\in GL(2,\mathbb Z)$ 
with $p_1,\,p_2,\,q_1,\,q_2\ge 0$.  Put $\tilde{\mathbf  p}=
\mathbf p\otimes I_{n-2}\in GL(n,\mathbb Z)$.  Then
\begin{align*}
 K_{(x_1,x_2),x\to x^{\tilde{\mathbf p}}}^{\mu,(\lambda_1,\lambda_2)}
x^{\mathbf m}
=\frac{\Gamma(\lambda_1)\Gamma(\lambda_2)}
  {\Gamma(\lambda_1+\lambda_2+\mu)}
  \frac{(\lambda_1)_{p_1m_1+p_2m_2}(\lambda_2)_{q_1m_1+q_2m_2}}
  {(\lambda_1+\lambda_2+\mu)_{(p_1+q_1)m_1+(p_2+q_2)m_2}}x^{\mathbf m},\\
 L_{(x_1,x_2),x\to x^{\tilde{\mathbf p}}}^{\mu,(\lambda_1,\lambda_2)}
x^{\mathbf m}
=\frac{\Gamma(\lambda_1+\lambda_2+\mu)}
  {\Gamma(\lambda_1)\Gamma(\lambda_2)}
   \frac{(\lambda_1+\lambda_2+\mu)_{(p_1+q_1)m_1+(p_2+q_2)m_2}}
  {(\lambda_1)_{p_1m_1+p_2m_2}(\lambda_2)_{q_1m_1+q_2m_2}}
  x^{\mathbf m}.
\end{align*}

Successive applications of these transformations to $(1-|\mathbf x|)^{-\lambda}$ or 
$(1-\mathbf x)^{-\boldsymbol\lambda}$ or $e^{|\mathbf x|}$,\ldots, we have many examples of integral representations of power
series whose coefficients of $\frac{x^\mathbf m}{\mathbf m!}$ are expressed by the quotient of
products of the form $(\lambda)_{p_1m_1+\cdots+p_nm_n}$.

The series 
\begin{equation}\label{eq:GF}
\phi(x,y)=\sum_{m=0}^\infty\sum_{n=0}^\infty
  \frac{\prod_{\nu=1}^K (a_\nu)_{m+n}\prod_{\nu=1}^M (b_\nu)_{m}\prod_{\nu=1}^N (c_\nu)_{n}}
  {\prod_{\nu=1}^{K'} (a'_\nu)_{m+n}\prod_{\nu=1}^{M'} (b'_\nu)_m\prod _{\nu=1}^{N'}(c'_\nu)_n}\frac{x^m}{m!}\frac{y^n}{n!}
\end{equation}
with the condition
\[
  (K+M)-(K'+M')=(K+N)-(K'+N')=1.
\]
is an example. Then Appell's hypergeometric functions $F_1$, $F_2$, $F_3$ and $F_4$ correspond to
$(K,M,N;K',N',N')=(1,1,1;1,0,0)$, $(1,1,1;0,1,1)$, $(0,2,2;1,0,0)$ and
$(2,0,0;0,1,1)$, respectively.
In general $\phi(x,y)$ may have several integral expressions as in the case of $F_1$ and $F_2$.
The series \eqref{eq:GF} with $M=M'+1$, $N=N'+1$ and $K=K$ is a generalization of Appell's $F_1$,
which will be given in \S\ref{sec:KZ}  as an example.

The series 
\begin{equation}\label{eq:211}
\begin{split}
  &K_x^{\gamma_2-\beta_2,\beta_2}\cdot (1-x)^{\alpha_2}\cdot K_x^{\gamma_1-\beta_1,\beta_1}
  (1-x)^{-\alpha_1}\\
 &\quad =\frac{\Gamma(\beta_1)\Gamma(\beta_2)}{\Gamma(\gamma_1)\Gamma(\gamma_2)}
  \sum_{m=0}^\infty\sum_{n=0}^\infty \frac{(\alpha_1)_m(\alpha_2)_m(\beta_1)_m(\beta_2)_{m+n}}
  {(\gamma_1)_m(\gamma_2)_{m+n}m!n!}x^{m+n}
\end{split}
\end{equation}
of $x\in\mathbb C$ satisfies a Fuchsian differential equation with the spectral type $211,211,211$ 
(cf.~\cite[\S13.7.5]{Ow} and Remark~\ref{rm:1-x}) and the coefficients of $x^k$ is not simple.
%\end{example}
%%%%%%%%%%%%%%%%%%%%%%  Deq %%%%%%%%%%%%%%%%%%%%%%%%%%%%%%
\section{Differential equations}
In this section we examine the differential equations satisfied by our invertible 
integral transformations of a function $u(x)$ in terms of the differential equation 
satisfied by $u(x)$.
We denote by $W[x]$ the ring of differential operators with polynomial coefficients
and put  $W(x)=\mathbb C[x]\otimes W[x]$.  Then $W[x]$  is called a Weyl algebra.

First we review the related results in \cite{Ow}. The integral transformation 
$u\mapsto I_c^\mu u$  given by \eqref{eq:RL} satisfies 
\begin{align}
  I_c^{-\mu}\circ I_c^\mu&=\mathrm{id}\\
  I_c^\mu\circ \p &= \p\circ I_c^\mu
\text{ \ and \ }I_c^\mu\circ\vartheta = (\vartheta -\mu)\circ  I_c^\mu\label{eq:Icm}
\end{align}
under the notation
\begin{equation}
 \p=\tfrac{d}{dx},\quad \vartheta = x\p.
\end{equation}
Hence for an ordinary differential operator $P\in W[x]$, we define
the middle convolution $\mathrm{mc}_\mu(P)$ of $P$ by
\begin{equation}
  \mathrm{mc}_\mu(P) := \p^{-m}\sum_{i,j}a_{i,j}\p^i(\vartheta -\mu)^j\in W[x].
\end{equation}
Here we first choose a positive integer $k$ so that
\begin{equation}
  \p^k P=\sum_{i\ge0,\ j\ge0}a_{i,j}\p^i\vartheta^j\quad ([a_{i,j},x]=[a_{i,j},\p]=0)
\end{equation}
and then we choose the maximal positive integer $m$ so that  $\mathrm{mc}_\mu(P)\in W[x]$.
The number $k$ can be taken to be the degree of $P$ with respect to $x$. 
Then we have
\begin{equation}
   Pu=0\ \ \Rightarrow\ (\mathrm{mc}_\mu(P))I_c^\mu u=0.
\end{equation}

The transformation $u\mapsto f(x)u$ of $u(x)$ defined by a suitable function 
$f(x)$ induces an automorphism $\Ad(f)$ of $W(x)$.  
Namely $\Ad(f)$ is called an addition and defined by 
\begin{equation}
  \Ad(f)\p=\p-\tfrac{\p(f)}{f}\text{ \ and \ }\Ad(f)x=x.
\end{equation}
Hence $\tfrac{\p(f)}{f}$ should be a rational function.  Then $f(x)$ can be a function
$(x-c)^\lambda$ or $f(x)=e^{r(x)}$ with a rational function $r(x)$.

There is another transformation $\mathrm{R}P$ of $P\in W(x)\setminus\{0\}$ 
where we define $\mathrm{R}P=r(x)P$ with 
$r(x)\in\mathbb C[x]\setminus\{0\}$ so that $r(x)P\in W[x]$ has 
the minimal degree with respect to $x$.
Then $\mathrm{R}P$ is called the \emph{reduced representative} of $P$.
When we consider $\mathrm{mc}_\mu(P)$, we usually replace $P$ by $\mathrm{R}P$.

Let $Pu=0$ be a rigid Fuchsian differential equation on  $\mathbb P^1$.
Then it is proved in \cite{Ow} that $P$ is obtained by successive applications of
$\Ad(f)$ and $\mathrm{mc}_\mu\circ\mathrm{R}$ to $\p$ and hence we have an integral 
representation of the solution to this equation and moreover its expansion into a power 
series.

In a similar way, the author \cite[\S13.10]{Ow} examines Appell's hypergeometric functions 
using the integral transformation 
\begin{align}
  J^\mu_x(u)(x)
 &:=\int_\Delta (1-s_1x_1-\dots-s_nx_n)^\mu u(s_1,\dots,s_n)\D s\\
 &\,=\frac1{x_1\cdots x_n}\int_{\Delta'} (1-t_1-\dots-t_n)^{\mu} 
     u(\tfrac{t_1}{x_1},\dots,\tfrac{t_n}{x_n})\D t
   \quad(t_j=s_jx_j)\notag
\end{align}
with certain regions $\Delta$ and $\Delta'$ of integrations and
get integral representations of Appell's hypergeometric functions.
We show there the commuting relations 
\begin{align}\label{eq:Jcm}
\begin{split}
  J_x^\mu\circ \vartheta_j&=(-1-\vartheta_j)\circ J_x^\mu,\\
  J_x^\mu\circ \p_j&=x_j(\mu-\vartheta_1-\cdots-\vartheta_n)\circ J_x^\mu,
\end{split}
\end{align}
which correspond to \eqref{eq:Icm} and imply the following proposition
and then we get the differential equations satisfied by Appell's hypergeometric functions.

In general, we have the following proposition.
\begin{proposition}[{\cite[Proposition 13.2]{Ow}}]\label{prop:3}
For a differential operator
\[
  P=\sum_{\boldsymbol \alpha,\boldsymbol\beta \in\mathbb N^n} 
  c_{\boldsymbol \alpha,\boldsymbol \beta}\p^{\boldsymbol\alpha}\vartheta^{\boldsymbol\beta}
\]
we have
\[
\begin{split}
  J_x^\mu\bigl(Pu(x)\bigr)&=J_x^\mu(P)J_x^\mu\bigl(u(x)\bigr),\\
  J_x^\mu(P)&:=\sum_{\boldsymbol \alpha,\boldsymbol\beta \in\mathbb N^n} 
    c_{\boldsymbol\alpha,\boldsymbol\beta}
   \Bigl(\prod_{k=1}^n\bigl(x_k(\mu-\vartheta_1-\cdots-\vartheta_n)\bigr)^{\alpha_k}\Bigr)
   (-\boldsymbol\vartheta-1)^{\boldsymbol\beta}.
\end{split}
\]
\end{proposition}
Here the sums are finite and we use the notation.
\begin{align}
\p_x=\tfrac\p{\p x},\ \p_y=\tfrac\p{\p y},\ \vartheta_x=x\p_x,\ \vartheta_y=y\p_y,\ 
\p_i=\tfrac\p{\p x_i},\ \vartheta_i=x_i\p_i.
\end{align}

Comparing the definition of integral transformations we have the following.
\begin{proposition}\label{prop:4}
The integral transformations defined in \S\ref{sec:Int} is expressed by $J^\mu_x$ as follows.
\begin{align}
\begin{split}
   K_x^\mu&=\frac{1}{\Gamma(\mu)}
   T_{x\to(\tfrac1{x_1},\dots,\tfrac1{x_n})}\circ x_1\cdots x_n\cdot J_x^{\mu-1}\\
  &\text{with}\quad \Delta'=\bigl\{(t_1,\dots,t_n)\mid t_1>0,\dots,t_n>0,\ t_1+\cdots+t_n<1\bigr\}
\end{split}
\intertext{and}
\begin{split}
  L_x^\mu&=\frac{\Gamma(\mu+n)}{(2\pi i)^n}
   J_x^{-\mu-n}\circ T_{x\to(\tfrac1{x_1},\dots,\tfrac1{x_n})}\circ  x_1\cdots x_n\\
  &\text{with}\quad \Delta'=\bigl\{(t_1,\dots,t_n)\mid \RE t_1=\cdots=\RE t_n=\tfrac1{n+1}\bigr\}.
\end{split}
\end{align}
\end{proposition}

For $\mathbf p\in GL(n,\mathbb Z)$ we put $\mathbf q=\mathbf p^{-1}$. Then
\begin{align*}
 T_{x\to x^{\mathbf p}}(x_j)&=x^{p_ {*.j}}=\prod_{\nu=1}^n x_\nu^{p_\nu,j}\text{ \ and \ }
 T_{x\to x^{\mathbf p}}(\p_i) = \sum_{j=1}^nq_{i,j}\frac{x_j}{x^{p_{j,i}}}\p_j.
\end{align*}

In particular
\begin{align*}
 T_{x\to (\tfrac1{x_1},\dots,\tfrac1{x_n})}(x_j)&=\tfrac1{x_j},\quad
 T_{x\to (\tfrac1{x_1},\dots,\tfrac1{x_n})}(\p_j) = -x_j^2\p_j
\intertext{and}
 T_{x\to (\tfrac1{x_1},\dots,\tfrac1{x_n})}\circ x_1\dots x_n&=
 \frac1{x_1\cdots x_n}\circ T_{x\to (\tfrac1{x_1},\dots,\tfrac1{x_n})}
% T_{x\to (\tfrac1{x_1},\dots,\tfrac1{x_n})}\circ x_1\dots x_n\circ x_i
%  &=\tfrac1{x_i}\circ T_{x\to (\tfrac1{x_1},\dots,\tfrac1{x_n})}\circ x_1\dots x_n,\\
% T_{x\to (\tfrac1{x_1},\dots,\tfrac1{x_n})}\circ x_1\dots x_n\circ\p_i
%  &=-(x_i^2\p_i+x_i)\circ  T_{x\to (\tfrac1{x_1},\dots,\tfrac1{x_n})}\circ x_1\dots x_n.
\end{align*}
and thus  we have the following lemma.
\begin{lemma}\label{lem:1}
Defining 
\begin{align}
 \widetilde{u}(x_1,\dots,x_n)&:=\tfrac1{x_1\cdots x_n}u(\tfrac1{x_1},\dots,\tfrac1{x_n}),\notag\\
 \widetilde P=P\tilde{\phantom{p}} 
 &:=\sum a_{\boldsymbol\alpha}(\tfrac1{x_1},\ldots,\tfrac1{x_n})\prod_\nu (-x_\nu^2\p_\nu-x_\nu)^{\beta_\nu }\\
\text{for \ \ }P&=\sum a_{\boldsymbol\alpha}(x)\p^{\boldsymbol\alpha}\in W(x),\notag
\end{align}
we have
\begin{gather}
 \widetilde{P u}=\widetilde P\widetilde u,\\
 \widetilde{x_j}=\tfrac1{x_j},\ \widetilde{\p_j}=-x_j^2\p_j-x_j=-x_j(\vartheta_j+1),\ \widetilde\vartheta_j=-\vartheta_j-1.
\end{gather}
\end{lemma}
Hence Proposition~\ref{prop:3}, Proposition~\ref{prop:4} and Lemma~\ref{lem:1} show
\begin{align}
\begin{split}
 K_x^\mu\circ \vartheta_j&
 =\widetilde{(-1-\vartheta_j)}\circ K_x^\mu
%\\&T_{x\to(\tfrac1{x_1},\dots,\tfrac1{x_n})}(-\vartheta_i)\circ K_x^\mu
 =\vartheta_j\circ K_x^\mu,
\end{split}\label{eq:Ktheta}\\
\begin{split}
 K_x^\mu\circ\p_j&=\bigl(
 x_j(\mu-1-\vartheta_1-\cdots-\vartheta_n)\bigr)\!\tilde{\phantom l}\circ K_x^\mu\\
%  &=T_{x\to(\tfrac1{x_1},\dots,\tfrac1{x_n})}
%   (x_j\bigl(\mu-1-\vartheta_1-\cdots-\vartheta_n+n)\bigr)
%   \circ K_x^\mu\\
  &=\tfrac1{x_j}(\vartheta_1+\cdots+\vartheta_n+\mu+n-1)\circ K_x^\mu.
\end{split}\label{eq:Kdelta}
\end{align}
Similarly we have $L_x^\mu\circ\vartheta_i=\vartheta_i\circ L_x^\mu$ and
\begin{align}
%T_{x\to(\tfrac1{x_1},\dots,\tfrac1{x_n})}(\p_j)\circ x_1\cdots x_n&=
%-x_j^2\p_j\circ x_1\cdots x_n=-x_1\cdots x_n\circ (x_j^2\p_j+x_j)\notag\\
\begin{split}
 L_x^\mu\circ  x_j(\vartheta_j+1)
 &= (x_j(\vartheta_1+\cdots+\vartheta_n+\mu+n))\circ
 L_x^\mu.
\end{split}
\end{align}

These relations can be checked by applying them to $x^{\boldsymbol\alpha}$. For example,
it follows from \eqref{eq:Lx0} that
\begin{align*}
 L_x^\mu\circ x_j(\vartheta_j+1)x^{\boldsymbol \alpha}
  &=L_x^\mu (\alpha_j+1)x_jx^{\boldsymbol\alpha}\\
  &=(\alpha_j+1)\frac{\Gamma(|\boldsymbol\alpha|+\mu+n+1)}{\Gamma(\alpha_1+1)\cdots\Gamma(\alpha_j+2)\cdots\Gamma(\alpha_n+1)}x_jx^{\boldsymbol\alpha}\\
  &=\frac{(|\boldsymbol\alpha|+\mu+n)\Gamma(|\boldsymbol\alpha|+\mu+n)}
    {\Gamma(\boldsymbol\alpha+1)}x_jx^{\boldsymbol\alpha}\\
  &=x_j(\vartheta_1+\cdots+\vartheta_n+\mu+n) L_x^\mu(x^{\boldsymbol\alpha}).
\end{align*}
We also note that \eqref{eq:Ktheta} is directly given by
\begin{align*}
(\vartheta_iK_x^\mu u)(x)&=\frac1{\Gamma(\mu)}\int_0^1 (1-|\mathbf t|)^{\mu-1} t_ix_i (\p_iu)(tx)\D t
 =(K_x^\mu\vartheta_i u)(x)
\end{align*}
and the equality
\begin{align*}
\tfrac{\p}{\p{t_i}}\bigl((1-|\mathbf t|)^{\mu-1} u(tx)\bigr)&=-(\mu-1)(1-|\mathbf t|)^{\mu-2} u(tx)+(1-t)^{\mu-1}x_i(\p_i u)(tx)
\intertext{shows}
x_i K_x^\mu\p_i&=(\mu-1) K_x^{\mu-1}
\intertext{and therefore}
 \mu\int_0^1 
  (1-|\mathbf t|)^{\mu-1} u(tx)\D t&=x_i\int_0^1 (1-|\mathbf t|)^{\mu}(\p_i u)(tx)\D t\\
  &=x_i\int_0^1 (1-|\mathbf t|)^{\mu-1}(1-t_1-\cdots-t_n)(\p_i u)(tx)\D t,\allowdisplaybreaks\\
x_i K_x^\mu \p_i u&=\mu K_x^\mu u + \sum_{\nu=1}^n\frac1{\Gamma(\mu)}\frac{x_i}{x_\nu}\int_0^1 
    (1-|\mathbf t|)^{\mu-1}(x_\nu \p_i u)(tx)\D t\\
 &=\mu K_x^\mu u +\sum_{\nu=1}^n \frac{x_i}{x_\nu} K_x^\mu \p_i x_\nu u - K_x^\mu u\allowdisplaybreaks\\
 &=(\mu-1) K_x^\mu u +\sum_{\nu=1}^n K_x^\mu \p_\nu x_\nu u\allowdisplaybreaks\\
 &=(\mu+n-1) K_x^\mu u +\sum_{\nu=1}^n \vartheta_\nu K_x^\mu u,
%\\ K_x^\mu \p_i&=\tfrac1{x_i}(\vartheta_1+\cdots+\vartheta_n+\mu+n-1)K_x^\mu,
\end{align*}
which implies \eqref{eq:Kdelta}.

Thus we have the following theorem.
\begin{theorem}\label{thm:PKL}
Suppose $u(x)$ satisfies $Pu(x)=0$ with a certain $P\in W(x)$.

{\rm i)} \ Putting
\begin{align}
  Q&=\mathrm{R}P=\sum_{\boldsymbol\alpha,\,\boldsymbol\beta\in\mathbb N^n}
     a_{\boldsymbol\alpha,\boldsymbol\beta}x^{\boldsymbol\alpha}
 \p^{\boldsymbol\beta}.%\quad(a_{\boldsymbol\alpha,\boldsymbol\beta}\in\mathbb C),
\intertext{we choose $\boldsymbol\gamma\in\mathbb Z^n$ so that}
\label{eq:Qdtheta}
  \p^{\boldsymbol\gamma} Q&=\sum_{\boldsymbol\alpha,\,\boldsymbol\beta\in\mathbb N^n}
  c_{\boldsymbol\alpha,\boldsymbol\beta}\p^{\boldsymbol\alpha}
  \vartheta^{\boldsymbol\beta}\quad(c_{\boldsymbol\alpha,\boldsymbol\beta}\in\mathbb C).
\end{align}
Then we have $K_x^\mu(\p^\gamma Q) K_x^\mu u(x)=0$ with
\begin{equation}
 K_x^\mu(\sum c_{\boldsymbol\alpha,\boldsymbol\beta}\p^{\boldsymbol\alpha}
  \vartheta^{\boldsymbol\beta})
 :=\mathrm{R}\sum c_{\boldsymbol\alpha,\boldsymbol\beta}\Bigl(\prod_{k=1}^n
 (\tfrac1{x_k}(\vartheta_1+\cdots+\vartheta_n+\mu+n-1))^{\alpha_k}\Bigr)
 \vartheta^{\boldsymbol\beta}.
\end{equation}

{\rm ii)} \ 
Putting 
\begin{equation}
  Q=\mathrm{R}\tilde P,
\end{equation}
we choose $\boldsymbol\gamma\in\mathbb Z^n$ so that \eqref{eq:Qdtheta} holds.
Then we have $L_x^{\mu}(\p^{\boldsymbol\gamma}Q) L_x^\mu u(x)=0$ with
\begin{equation}
 L_x^\mu(\sum c_{\boldsymbol\alpha,\boldsymbol\beta}\p^{\boldsymbol\alpha}
  \vartheta^{\boldsymbol\beta})
  :=\mathrm{R}\sum c_{\boldsymbol\alpha,\boldsymbol\beta}\Bigl(\prod_{k=1}^n
 \bigl(x_k(\mu-\vartheta_1-\cdots-\vartheta_n)\bigr)^{\alpha_k}\Bigr)
  (-\boldsymbol\vartheta-1)^{\boldsymbol\beta}.
\end{equation}
\end{theorem}
%%%%%%%%%%%%%
\begin{remark} 
i) \ In Theorem~\ref{thm:PKL} i), $\boldsymbol\gamma=(\gamma_1,\dots,\gamma_n)$ can be
taken by 
\begin{align*}
  \gamma_j&=\max\{0,\alpha_j-\beta_j|a_{(\alpha_1,\dots,\alpha_n),(\beta_1,\dots,\beta_n)}\ne0\}\qquad(1\le j\le n).
\end{align*}

ii) \ 
If $P\in W[x,y]$ in Theorem~\ref{thm:PKL}, it is clear that the theorem is valid 
under the assumption $c_{\boldsymbol\alpha,\boldsymbol\beta}\in W[y]$.
%\end{remark}

%\begin{remark}
iii) \  
Suppose $u(x)\in\mathcal O_0$ satisfies  $P_1P_2u(x)=0$ with $P_1,\,P_2\in W[x]$
and $\{u\in\mathcal O_0\mid P_1u=0\}=\{0\}$.  Then $P_2u(x)=0$.

iv) \ 
Without the assumption \eqref{eq:positive} 
we can define transformations $K_{x,x\to x^{\mathbf p}}^{\mu,\boldsymbol\lambda}$ and
$L_{x,x\to x^{\mathbf p}}^{\mu,\boldsymbol\lambda}$ on $\mathcal O_0$ by 
\eqref{eq:Kco} and \eqref{eq:Lco}. 
Even in this case the results in this section are clearly valid.
\end{remark}

We will calculate some examples.  By the integral expression
\begin{align*}
 &F_1(\lambda_0,\lambda_1,\lambda_2,\mu;x,y)
 =\frac{\Gamma(\mu)}{\Gamma(\lambda_1)\Gamma(\lambda_2)}
  K_{x,y}^{\mu-\lambda_1-\lambda_2,\lambda_1,\lambda_2}(1-x-y)^{-\lambda_0}\\
 &\quad =C_1% \frac{\Gamma(\mu)}{\Gamma(\lambda_1)\Gamma(\lambda_2)\Gamma(\mu-\lambda_1-\ambda_2)}
   \int_{\substack{s>0,\,t>0\\ s+t<1}}
   (1-s-t)^{\mu-\lambda_1-\lambda_2-1}s^{\lambda_1-1}t^{\lambda_2-1}(1-sx-ty)^{-\lambda_0}\D s\D t
\end{align*}
corresponding to  \eqref{eq:FD} and \eqref{eq:Kdef},
we calculate the system of differential equations satisfied by 
$F_1(\lambda_0,\lambda_1,\lambda_2,\mu;x,y)$ as follows.
Putting
\begin{align*}
 h&:=x^{\lambda_1-1}y^{\lambda_2-1}(1-x-y)^{-\lambda_0},
\intertext{we have}
 \Ad(h)\p_x&=\p_x-\tfrac{\lambda_1-1}{x}-\tfrac{\lambda_0}{1-x-y},\\
 \Ad(h)\p_y&=\p_y-\tfrac{\lambda_2-1}{y}-\tfrac{\lambda_0}{1-x-y},\\
\Ad(h)(\vartheta_x+\vartheta_y)&=\vartheta_x+\vartheta_y-\tfrac{\lambda_0}{1-x-y}
 -(\lambda_1+\lambda_2-\lambda_0-2),\\
\Ad(h)(\vartheta_x+\vartheta_y-\p_x)&=\vartheta_x+\vartheta_y-\p_x+\tfrac{\lambda_1-1}x-(\lambda_1+\lambda_2-\lambda_0-2).
\intertext{Hence we put}
Q\ &\!\!:=\mathrm{R}\Ad(h)(\vartheta_x+\vartheta_y-\p_x)\\
 &=(\vartheta_x+1)(\vartheta_x+\vartheta_y -\lambda_1-\lambda_2+\lambda_0+2)
 -\p_x \vartheta_x+\lambda_1-1
\intertext{and we have}
K_x^{\mu-\lambda_1-\lambda_2}(Q)&=x(\vartheta_x+1)(\vartheta_x+\vartheta_y -\lambda_1-\lambda_2+\lambda_0+2)\\
 &\quad{} -(\vartheta_x+\vartheta_y+\mu-\lambda_1-\lambda_2+1)(\vartheta_x-\lambda_1+1),\\
\Ad(x^{1-\lambda_1}y^{1-\lambda_2})K_x^\mu(Q)&=x(\vartheta_x+\lambda_1)
   (\vartheta_x+\vartheta_y+\lambda_0)
   -(\vartheta_x+\vartheta_y+\mu-1)\vartheta_x\\
 &=x\bigl((\vartheta_x+\lambda_1)(\vartheta_x+\vartheta_y+\lambda_0)
   -\p_x(\vartheta_x+\vartheta_y+\mu-1)\bigr).
\end{align*}
Hence $F_1(\lambda_0,\lambda_1,\lambda_2,\mu;x,y)$ is a solution of the system
\begin{align}\label{eq:F1K}
\begin{cases}
 (\vartheta_x+\lambda_1)(\vartheta_x+\vartheta_y+\lambda_0)
   -\p_x(\vartheta_x+\vartheta_y+\mu-1)u_1=0,\\
 (\vartheta_y+\lambda_2)(\vartheta_x+\vartheta_y+\lambda_0)
   -\p_y(\vartheta_x+\vartheta_y+\mu-1)u_1=0.
\end{cases}
\end{align}
Next we consider the integral representation
\begin{align*}
 F_1(\lambda_0,\lambda_1,\lambda_2,\mu;x,y)&=\frac{\Gamma(\mu)}{\Gamma(\lambda_0)}
   K^{\mu-\lambda_0,\lambda_0}_{x,(x,y)\to(x,\tfrac xy)}(1-x)^{-\lambda_1}(1-y)^{-\lambda_2}
\\
  &=C_1'\int_0^1t^{\lambda_0-1}(1-t)^{\mu-\lambda_0}
  (1-tx)^{-\lambda_1}(1-ty)^{-\lambda_2}\D t
\end{align*}
\if0
Putting
\begin{align*}
  h&=x^{\lambda_0-1}(1-x)^{-\lambda_1}(1-\tfrac xy)^{-\lambda_2},
\intertext{we have}
 \Ad(h)\p_x&=\p_x-\tfrac{\lambda_0-1}x-\tfrac{\lambda_1}{1-x} -\tfrac{\lambda_2}{y-x},\\
 \Ad(h)\p_y&=\p_y-\tfrac xy\tfrac{\lambda_2}{y-x},\\
 Q&:=\mathrm{R}\Ad(h)(\vartheta_x-\vartheta_y)=\mathrm{R}(\vartheta_x-\vartheta_y
  -(\lambda_0-1)-\tfrac{\lambda_1 x}{1-x})\\
          &=(1-x)(\vartheta_x-\vartheta_y)-(\lambda_0-1)-(\lambda_0+\lambda_1-1)x.\\
 \mathrm{R}K_x^{\mu-\lambda_0}(\p_xQ)&=
   \mathrm{R}K_x^{\mu-\lambda_0}\bigl(
    (\p_x-\vartheta_x-1)(\vartheta_x-\vartheta_y)-(\lambda_0-1)\p_x-(\lambda_0+\lambda_1-1)
     (\vartheta_x+1)\bigr)\\
  &=\mathrm{R}\bigl((\p_x+\tfrac\mu x-\vartheta_x-1)(\vartheta_x-\vartheta_y)-
  (\lambda_0-1)(\p_x+\tfrac\mu x) -(\lambda_0+\lambda_1-1)
     (\vartheta_x+1)\bigr)\\
 &=(\vartheta_x+\mu)(\vartheta_x-\vartheta_y-\lambda_0+1)-x(\vartheta_x+1)(\vartheta_x-\vartheta_y
 +\lambda_0+\lambda_1+1)
\end{align*}
\fi
corresponding to \eqref{eq:FD2}. 
Since
\begin{align*}
 T_{(x,y)\mapsto(x,\tfrac xy)}&:\p_x\mapsto \tfrac1x(\vartheta_x+\vartheta_y),\  
   \p_y\mapsto -\tfrac yx\vartheta_y,\ \vartheta_x\mapsto\vartheta_x+\vartheta_y,\ 
   \vartheta_y\mapsto -\vartheta_y,
\end{align*}
we have 
\begin{align*}
\p_x&\xrightarrow{\Ad((1-x)^{-\lambda_1}(1-y)^{-\lambda_2})}
   \p_x -\tfrac{\lambda_1}{1-x}
%\\&
  \ \xrightarrow{T_{(x,y)\to(x,\frac xy)}}
   \tfrac1x(\vartheta_x+\vartheta_y)-\tfrac{\lambda_1}{1-x}\\
 &\xrightarrow{\Ad(x^{\lambda_0-1})}\ 
  \tfrac1x(\vartheta_x+\vartheta_y)-\tfrac{\lambda_0-1}x-\tfrac{\lambda_1}{1-x}
 \\&
 \xrightarrow{\mathrm{R}}(1-x)(\vartheta_x+\vartheta_y-\lambda_0+1)-\lambda_1x\\
 &\xrightarrow{\p_x}
   (\p_x-\vartheta_x-1)(\vartheta_x+\vartheta_y-\lambda_0+1)-\lambda_1(\vartheta_x+1)\\
 &\xrightarrow{K_x^{\mu-\lambda_0}}
   (\p_x+\tfrac{\mu-\lambda_0}x-\vartheta_x-1)
   (\vartheta_x+\vartheta_y-\lambda_0+1)-\lambda_1(\vartheta_x+1)\\
  &\xrightarrow{\Ad(x^{1-\lambda_0})}
  (\p_x+\tfrac{\mu-1}x-\vartheta_x-\lambda_0)(\vartheta_x+\vartheta_y)
   -\lambda_1(\vartheta_x-\lambda_0)\\
  &\xrightarrow{T_{(x,y)\to(x,\frac xy)}}
     (\tfrac1x(\vartheta_x+\vartheta_y)+\tfrac{\mu-1} x-\vartheta_x-\vartheta_y-\lambda_0)
    \vartheta_x
   -\lambda_1(\vartheta_x-\lambda_0)\\
 &=\p_x(\vartheta_x+\vartheta_y+\mu-1)-(\vartheta_x+\lambda_1)(\vartheta_x+\vartheta_y+\lambda_0).
\end{align*}
Thus we also get the system 
\eqref{eq:F1K} characterizing $F_1(\lambda_0,\lambda_1,\lambda_2,\mu;x,y)$.

We have similar calculations for other Appell's hypergeometric series as follows.
\begin{align*}
%%% F_2 %%%
&F_2(\lambda_0;\mu_1,\mu_2;\lambda_1,\lambda_2;x,y)
 =\frac{\Gamma(\lambda_1)\Gamma(\lambda_2)}{\Gamma(\mu_1)\Gamma(\mu_2)}
  K_x^{\lambda_1-\mu_1,\mu_1}K_y^{\lambda_2-\mu_2,\mu_2}(1-x-y)^{-\lambda_0}\\
&=C_2%\frac{\Gamma(\lambda_1)\Gamma(\lambda_2)}{\Gamma(\mu_1)\Gamma(\mu_2)}
 \int_0^1\int_0^1
 s^{\mu_1}t^{\mu_2}(1-s)^{\lambda_1-\mu_1-1}(1-t)^{\lambda_2-\mu_2-1}(1-sx-ty)^{-\lambda_0}
  \tfrac{\D s}s\tfrac{\D t}t,
\allowdisplaybreaks\\
&\p_x\xrightarrow
 {\Ad(x^{\mu_1-1}y^{\mu_2-1}(1-x)^{\lambda_1-\mu_1-1}(1-y)^{\lambda_2-\mu_2-1}}
 \p_x -\tfrac{\mu_1-1}x+\tfrac{\lambda_1-\mu_1-1}{1-x}\allowdisplaybreaks\\
&\xrightarrow{\textrm{R}}x(1-x)\p_x+(\lambda_1-2)x-(\mu_1-1)\allowdisplaybreaks\\
&\xrightarrow{\p_x}\p_x x(-\vartheta_x+\lambda_1-2)+\p_x(\vartheta_x-\mu_1+1)\\
&\xrightarrow{J_{x,y}^{-\lambda_0}}-\vartheta_x(\vartheta_x+1+\lambda_1-2)
  +x(-\lambda_0-\vartheta_x-\vartheta_y)(-1-\vartheta_x-\mu_1+1)\\
&\quad =x\Bigl((\vartheta_x+\mu_1)(\vartheta_x+\vartheta_y+\lambda_0)
  -\p_x(\vartheta_x+\lambda_1-1)\bigr)\allowdisplaybreaks
\intertext{and}
%%%% F_3 %%%
&F_3(\lambda_1,\lambda_2;\lambda_1',\lambda_2';\mu;x,y)=
 \frac{\Gamma(\mu)}{\Gamma(\lambda_1)\Gamma(\lambda_2)}K_{x,y}^{\mu-\lambda_1-\lambda_2}
 (1-x)^{-\lambda_1'}(1-y)^{-\lambda_2'}\\
&=C_3% \frac{\Gamma(\mu)}{\Gamma(\lambda_1)\Gamma(\lambda_2)}
 \int_{\substack{s>0,t>0\\s+t<1}}
 s^{\lambda_1}t^{\lambda_2}(1-s-t)^{\mu-1}(1-sx)^{-\lambda_1'}(1-ty)^{-\lambda_2'}
  \tfrac{\D s}s\tfrac{\D t}t,\\
&\p_x\xrightarrow
  {\mathrm{R}\Ad(x^{\lambda_1-1}y^{\lambda_2-1}(1-x)^{-\lambda_1'}(1-y)^{-\lambda_2'}}
 x(1-x)\p_x+(\lambda_1-\lambda_1'-1)x-(\lambda_1-1)\\
&\xrightarrow{\p_x}
 \p_x x(-\vartheta_x+\lambda_1-\lambda_1'-1)+\tfrac1x(\vartheta_x+\vartheta_y+\mu-\lambda_1'-\lambda_2-1)(\vartheta_x-\lambda_1+1)\\
&\xrightarrow{\Ad(x^{1-\lambda_1}y^{1-\lambda_2}}
 -(\vartheta_x+\lambda_1)(\vartheta_x+\lambda_1')+\p_x(\vartheta_x+\vartheta_y+\mu-1)\allowdisplaybreaks
\intertext{and}
%%% F_4 %%%
&F_4(\mu,\lambda_0;\lambda_1,\lambda_2;x,y)=\frac{\Gamma(\lambda_1)\Gamma(\lambda_2)}{\Gamma(\mu)}
 L_{x,y}^{\mu-\lambda_1-\lambda_2,\lambda_1,\lambda_2}(1-x-y)^{-\lambda_0}\\
&=C_4% \frac{\Gamma(\lambda_1)\Gamma(\lambda_2)}{\Gamma(\mu)}
  \int_{\frac13-i\infty}^{\frac13+i\infty}
  s^{\lambda_1}t^{\lambda_2}(1-s-t)^{\lambda_1+\lambda_2-\mu-2}(1-sx-ty)^{-\lambda_0}\tfrac{\D s}s\tfrac{\D t}t,\allowdisplaybreaks\\
&\p_x\xrightarrow{\Ad((1-x-y)^{-\lambda_0})}\p_x-\tfrac{\lambda_0}{1-x-y}\allowdisplaybreaks\\
&\p_x-\vartheta_x-\vartheta_y-\lambda_0\quad(\in\ker(1-x-y)^{-\lambda_0})\\
&\xrightarrow{\Ad(x^{\lambda_1-1}y^{\lambda_2-1})}
 \p_x-\tfrac{\lambda_1-1}x-(\vartheta_x+\vartheta_y+\lambda_0-\lambda_1+1-\lambda_2+1)\\
&\xrightarrow{\frac1{xy}T_{(x,y)\mapsto(\frac1x,\frac1y)}}
 -x(\vartheta_x+1)-(\lambda_1-1)-(\lambda_1-1)x+\vartheta_x+\vartheta_y-\lambda_0+\lambda_1+\lambda_2\allowdisplaybreaks\\
&\xrightarrow{\p_x}-\p_x x(\vartheta_x+\lambda_1)
  +\p_x(\vartheta_x+\vartheta_y-\lambda_0+\lambda_1+\lambda_2)\\
&\xrightarrow{\Ad(L_{x,y}^{\mu-\lambda_1-\lambda_2})}\vartheta_x(-\vartheta_x-1+\lambda_1)\\
&\quad{}+x(\lambda_1+\lambda_2-\mu-2-\vartheta_x-\vartheta_y)(-\vartheta_x-\vartheta_y-\lambda_0
+\lambda_1+\lambda_2\allowdisplaybreaks\\
&\xrightarrow{\p_x} -\p_x x(\vartheta_x+\lambda_1)+\p_x(\vartheta_x+\vartheta_y-\lambda_0+\lambda_1+\lambda_2)\allowdisplaybreaks\\
&\xrightarrow{L_{x,y}^{\mu-\lambda_1-\lambda_2}}\vartheta_x(-\vartheta_x-1-\lambda_1)\\
&\quad{}+x(\lambda_1+\lambda_2-\mu-2-\vartheta_x-\vartheta_y)(-\vartheta_x-\vartheta_y-\lambda_0+\lambda_1+\lambda_2-2)\allowdisplaybreaks\\
&\xrightarrow{\Ad(x^{1-\lambda_1}y^{1-\lambda_2})}
x\Bigl((\vartheta_x+\vartheta_y+\mu)(\vartheta_x+\vartheta_y+\lambda_0)
 -\p_x(\vartheta_x+\lambda_1-1)
\Bigr).
\end{align*}
Here $C_1$. $C_1'$, $C_2$, $C_3$ and $C_4$ are constants easily obtained from the integral 
formula in \S\ref{sec:Int} with putting $x=y=0$. Hence
\begin{align}
 \begin{cases}
  u_2=F_2(\lambda_0;\mu_1,\mu_2;\lambda_1,\lambda_2;x,y),\\
  u_3=F_3(\lambda_1,\lambda_2;\lambda_1',\lambda_2';\mu;x,y),\\
  u_4=F_4(\mu,\lambda_0; \lambda_1,\lambda_2;x,y)\\
 \end{cases}
\end{align}
are solutions of the system
\begin{align}
\begin{cases}
\bigl((\vartheta_x+\mu_1)(\vartheta_x+\vartheta_y+\lambda_0)-\p_x(\vartheta_x+\lambda_1-1)\bigr)u_2=0,\\
\bigl((\vartheta_y+\mu_2)(\vartheta_x+\vartheta_y+\lambda_0)-\p_y(\vartheta_y+\lambda_2-1)\bigr)u_2=0,
\end{cases}
\\
\begin{cases}
\bigl((\vartheta_x+\lambda_1)(\vartheta_x+\lambda_1')-\p_x(\vartheta_x+\vartheta_y+\mu-1)\bigr)u_3=0,\\
\bigl((\vartheta_y+\lambda_2)(\vartheta_x+\lambda_2')-\p_y(\vartheta_x+\vartheta_y+\mu-1)\bigr)u_3=0,
\end{cases}
\\
\begin{cases}
\bigl((\vartheta_x+\vartheta_y+\mu)(\vartheta_x+\vartheta_y+\lambda_0)-\p_x(\vartheta_x+\lambda_1-1)\bigr)u_4=0,\\
\bigl((\vartheta_x+\vartheta_y+\mu)(\vartheta_x+\vartheta_y+\lambda_0)-\p_y(\vartheta_y+\lambda_2-1)\bigr)u_4=0.
\end{cases}
\end{align}
\begin{remark}
The above systems are directly obtained from the adjacent relations of the coefficients of 
Appell's hypergeometric series. Here we get them by the transformations of systems of 
differential equations corresponding to integral transformations of functions discussed 
in this paper so that it can be applied to general cases.
\end{remark}
%%%%%%%%%%%%%%%%%%%%%%%%%%%%%%%%%%%%%%%%%
%%%%%%%%%%%%%%%%%%  KZ %%%%%%%%%%%%%%%%%%
%%%%%%%%%%%%%%%%%%%%%%%%%%%%%%%%%%%%%%%%%
\section{KZ equations}\label{sec:KZ}
A Pfaffian system
\begin{align}
 \D u&=\sum_{0\le i<j\le q}\!\!A_{i,j}\frac{\D(x_i-x_j)}{x_i-x_j}u
\end{align}
with an unknown $N$ vector $u$ and constant square matrices $A_{i,j}$ of size $N$ 
is called a KZ (Knizhnik-Zamolodchikov type)
equation of rank $N$ (cf.~\cite{KZ}), which equals the system of the equations
\begin{align}
 \mathcal M\,:\,\frac{\partial u}{\partial x_i}&=\sum_{\substack{0\le\nu\le q\\ \nu\ne i}}
\frac{A_{i,\nu}}{x_i-x_\nu}u
 \qquad(i=0,\dots,q)\label{eq:KZ}
\end{align}
with denoting $A_{j,i}=A_{i,j}$.
The matrix $A_{i,j}$ is called the \emph{residue matrix} of $\mathcal M$ at $x_i=x_j$.
Here we always assume the \emph{integrability condition}
\begin{equation}\label{eq:Com0}
 \begin{cases}
 [A_{i,j},A_{k,\ell}]=0&(\forall\{i,j,k,\ell\}\subset
 \{0,\dots,q\}),\\
 [A_{i,j},A_{i,k}+A_{j,k}]=0&(\forall\{i,j,k\}\subset\{0,\dots,q\}),
 \end{cases}
\end{equation}
which follows from the condition $\D\D u=0$.
Here $i,\,j,\,k,\,\ell$ are mutually different indices:
%%%%
\begin{definition} Using the notation
\begin{align*}
 A_{i,i}&=A_\emptyset=A_i=0,\quad A_{i,j}=A_{j,i}\quad(i,\,j\in\{0,1,\dots,q+1\}),
  \allowdisplaybreaks\\
 A_{i,q+1}&:=-\sum_{\nu=0}^n A_{i,\nu},\allowdisplaybreaks\\
 A_{i_1,i_2,\dots,i_k}&:=\sum_{1\le \nu<\nu'\le k}A_{i_\nu,i_{\nu'}}\quad
(\{i_1,\dots,i_k\}\subset\{0,\dots,q+1\}), 
\end{align*}
we have
\begin{equation}
 [A_I,A_J]=0\quad\text{if }I\cap J=\emptyset\text{ or }I\subset J\text{ with }
 I,\,J\subset \{0,\dots,q+1\}.
\end{equation}

The matrix $A_{i,j}$ is called the residue matrix of $\mathcal M$ at
$x_i=x_j$ and $x_{q+1}$ corresponds to $\infty$ in $P^1_{\mathbb C}$.
\end{definition}
%%%

We note that any rigid irreducible Fuchsian system
\begin{equation}\label{eq:ODE}
 \mathcal N : \frac{\D u}{\D x}=\sum_{i=1}^q \frac{B_i}{x-x_i}u
\end{equation}
can be extended to KZ equation $\mathcal M$ with $x=x_0$ and $B_i=A_{0,i}$, which follows 
from the result by Haraoka~\cite{Ha} extending a middle convolution on KZ equations.

%%%
We assume that $\mathcal M$ is irreducible at a generic value of the holomorphic parameter 
contained in $\mathcal M$. Then it is shown in 
\cite[\S1]{Okz} that $A_{0,\dots,q}$ is a scalar matrix $\kappa I_N$ 
with $\kappa\in\mathbb C$ and by the gage transformation $u\mapsto (x_{q-1}-x_q)^{-\kappa} u$
we may assume that $\mathcal M$ is \emph{homogeneous}, which means
\begin{equation}\label{eq:homog}
 A_{i_0,\dots,i_q}=0\quad(0\le i_0<i_1<\dots<i_q\le q+1).
\end{equation}
Then the symmetric group $\mathfrak S_{q+2}$ which is identified with the permutation group
of the set of indices $\{0,1,\dots,q+1\}$ 
naturally acts on the space of the homogeneous KZ equations (cf.~\cite[\S6]{Okz}) :

\begin{tikzpicture}
 [root/.style={draw,circle,inner sep=0mm,minimum size=3mm}]
\node[root] (a0) at (0,0) {};
\node at (0,0.4) {$x_0$};
\node at (0,-0.4) {$x$};
\node[root] (a1) at (1.5,0) {};
\node at (1.5,0.4)  {$x_1$};
\node at (1.5,-0.4) {$y_1$};
\node[root] (a2) at (3,0) {};
\node at (3,0.4)  {$x_2$};
\node at (3,-0.4) {$y_2$};
\node[root] (b) at (6,0) {};
\node at (6,0.4)  {$x_{q-2}$};
\node at (6,-0.4) {$y_{q-2}$};
\node[root] (b0) at (7.5,0) {};
\node at (7.5,0.4)  {$x_{q-1}$};
\node at (7.5,-0.4) {$1$};
\node[root] (b1) at (9,0) {};
\node at (9,0.4) {$x_{q}$};
\node at (9,-0.4) {$0$};
\node[root] (b2) at (10.5,0) {};
\node at (10.5,0.4) {$x_{q+1}$};
\node at (10.5,-0.4) {$\infty$};
\draw (a0)--(a1)--(a2)--(4,0)  (5,0)--(b)--(b0)--(b1)--(b2);
\draw[dotted] (4,0)--(5,0);
\end{tikzpicture}

$(0,1)$\qquad\quad\ \  : \ $x\ \leftrightarrow \, y_1$,

$(i,i+1)$\qquad\ \,: \ $y_i\leftrightarrow \, y_{i+1}\qquad(1\le i\le q-3)$,

$(q-2,q-1)$ \,: \ $(x,y_1,\dots,y_{q-1},y_{q-2})\leftrightarrow
 (\tfrac{x}{y_{q-2}},\tfrac{y_1}{y_{q-2}},\dots,\tfrac{y_{q-1}}{y_{q-2}},\frac1{y_{q-2}})$,

$(q-1,q)$\qquad : \ $(x,y_1,\dots,y_{q-1},y_{q-2})\leftrightarrow(1-x,1-y_1,\dots,1-y_{q-1},1-y_{q-2})$,

$(q,q+1)$\qquad : \ $(x,y_1,\dots,y_{q-1},y_{q-2})\leftrightarrow(\tfrac1x,\tfrac1{y_1},\dots,
\tfrac1{y_{q-1}},\tfrac1{y_{q-2}})$.

\noindent
Here we put $(x_0,\dots,x_{q+1})=(x,y_1,\dots,y_{q-1},1,0,\infty)$ by a transformation 
$\mathbb P^1\ni x\mapsto ax+b$ which keeps the residue matrices $A_{i,j}$.

\medskip
For simplicity we assume $n=q-1=2$ and put $(x_0,x_1,x_2,x_3,x_4)=(x,y,1,0,\infty)$. 
Then \eqref{eq:homog} means
%$n=q-1=2$ :\quad %$(x_0,x_1,x_2,x_3,x_4)\to(x,y,1,0,\infty)$\\
\begin{equation}
A_{01}+A_{01}+A_{03}+A_{12}+A_{13}+A_{23}=0
\end{equation}
and the five residue matrices $A_{01}$, $A_{01}$, $A_{03}$, $A_{12}$ and $A_{13}$ uniquely 
determine the other five residue matrices $A_{23}$ and $A_{i4}$ with $0\le i\le 3$ 
and the action of $\mathfrak S_5$ is generated by the 4 involutions
\begin{align*}
(x_0,x_1,x_2,x_3,x_4)  \qquad& \to \qquad(x,y,1,0,\infty),\\
x_0\leftrightarrow x_1 \qquad& \to \qquad(x,y)\leftrightarrow(y,x),\\
x_1\leftrightarrow x_2 \qquad& \to \qquad (x,y)\leftrightarrow(\tfrac xy,\tfrac1y),\\ 
x_2\leftrightarrow x_3 \qquad& \to \qquad(x,y)\leftrightarrow(1-x,1-y),\\
x_3\leftrightarrow x_4 \qquad& \to \qquad(x,y)\leftrightarrow(\tfrac1x,\tfrac1y).
\end{align*}
%%%
In particular, the KZ system is determined by the equation
\begin{align}\label{eq:KZxy}
  \mathcal M : 
  \begin{cases}
   \displaystyle\frac{\p u}{\p x}=
   \frac{A_{01}}{x-y}u+\frac{A_{02}}{x-1} u + \frac{A_{03}}{x} u,\\
   \displaystyle\frac{\p u}{\p y}=
   \frac{A_{01}}{y-x}u+\frac{A_{12}}{y-1} u + \frac{A_{13}}{y} u.
  \end{cases}
\end{align}
\begin{remark}\label{rem:blowup}
The coordinate transformations corresponding to the involutions
\begin{align*}
 (x_0,x_1,x_2,x_3,x_4)\leftrightarrow   (x_2,x_1,x_0,x_4,x_3)  \qquad& \to \qquad
  (x,y)\leftrightarrow (x,\tfrac xy)\\
 (x_0,x_1,x_2,x_3,x_4)\leftrightarrow   (x_0,x_2,x_1,x_4,x_3)  \qquad& \to \qquad
  (x,y)\leftrightarrow (\tfrac yx,y)
\end{align*}
give the local coordinates of the blowing up of the singularities of the equation 
\eqref{eq:KZxy} at the origin:
\smallskip

%\vspace{1cm}

\begin{tikzpicture}
 [root/.style={draw,circle,inner sep=0mm,minimum size=1.8mm}]
\node[root] (a1) at (0,0) {};
\node[root] (a2) at (1.6,0) {};
\node[root] (a3) at (3.2,0) {};
\node[root] (a4) at (4.8,0) {};
\node at (-1,-0.35) {$(x,y)\mapsto$};
\node at (-1,0.25) {$\mathfrak S_5\ni$};
\node at (0,-0.35)  {$(y,x)$};
\node at (0, 0.25)  {$x_0\leftrightarrow x_1$};
\node at (1.6,-0.35)  {$(\frac xy,\frac 1y)$};
\node at (1.6, 0.25)  {$x_1\leftrightarrow x_2$};
\node at (3.2,-0.35)  {$(1-x,1-y)$};
\node at (3.2, 0.25)  {$x_2\leftrightarrow x_3$};
\node at (4.8,-0.35)  {$(\frac1x,\frac1y)$};
\node at (4.8, 0.25)  {$x_3\leftrightarrow x_4$};
\draw (a1)--(a2)--(a3)--(a4);
\end{tikzpicture}

\vspace{-1.3cm}\hspace{7.6cm}
\if\ORG \ \ \fi
\begin{tikzpicture}
\path[fill=black!40] (0,0)--(1/3,-1/10)--(1/3,1/2) --cycle;
\path[fill=black!40] (0,2.15)--(1/3,2.15)--(1/3,3/4) --(0,3/4) --cycle;
\draw (0,-0.1)--(0,2.3) (1,-0.1)--(1,2.3)  (2,-0.1)--(2,2.3)
      (-0.1,0)-- (2.3,0)  (-0.1,1)-- (2.3,1)  (-0.1,2)-- (2.3,2)
      (-0.1,-0.1)--(2.3,2.3);
\node at (-0.5,0.1) {$y=0$};
\node at (-0.5,1.1) {$y=1$};
\node at (-0.5,2.1) {$y=\infty$};
\node at (0,-0.3) {$x=0$};
\node at (1,-0.3) {$x=1$};
\node at (2,-0.3) {$x=\infty$};
\node at (1.5,1.4) {$x=y$};
\node at (0.45,1.7) {$\downarrow$};
\node at (0.45,0.2) {$\uparrow$};
\end{tikzpicture}

\vspace{-1.5cm}
\hspace{2.7cm}
$(x,y)\ \leftrightarrow\ (x,\frac xy) $

\hspace{0.5cm} 
$\{|x|<\epsilon,\ |y|<C|x|\}\ \leftrightarrow\  \{|x|<\epsilon,\ |y|>C^{-1}\}$

\hspace{2.05cm}
$x=y=0\ \leftrightarrow\ x=0$
\end{remark}

Now we review the result in \cite{DR, DR2, Ha} by using the transformations defined in this paper.
The convolution of the KZ equation \eqref{eq:KZ} corresponds to the transformation defined by 
\[
 (\tilde{\mathrm{mc}}_\mu u)(x,y):=\begin{pmatrix}
  I_{x,0}^{\mu+1} \frac{u(x,y)}{x-y}\\
  I_{x,0}^{\mu+1} \frac{u(x,y)}{y}\\
  I_{x,0}^{\mu+1} \frac{u(x,y)}{x}
\end{pmatrix}=
\begin{pmatrix}
    \tfrac1{\Gamma(\mu+1)}\int_0^x (1-t)^\mu \frac{u(t,y)}{t-y}\D t\\
    \tfrac1{\Gamma(\mu+1)}\int_0^x (1-t)^\mu \frac{u(t,y)}{y}\D t\\
    \tfrac1{\Gamma(\mu+1)}\int_0^x (1-t)^\mu \frac{u(t,y)}{t}\D t
\end{pmatrix}.
\]
We put $\tilde K_x^\mu = x^{-\mu}\circ \tilde{\mathrm{mc}}_\mu$ and 
$\tilde K_x^{\mu,\lambda}=x^{-\lambda}\circ \tilde K_x^\mu\circ x^{\lambda}$.
Then
\begin{align}\label{eq:IKZ}
 (\tilde K_x^{\mu,\lambda} u)(x,y)
 &=\begin{pmatrix} K_x^{\mu+1,\lambda}\frac {xu(x,y)}{x-y}\\
   K_x^{\mu+1,\lambda}\frac {xu(x,y)}{x-1}\\
   K_x^{\mu+1,\lambda} u(x,y)\\
\end{pmatrix}.
\end{align}
Putting $\tilde u=\tilde K_x^{\mu,\lambda} u$ for a solution $u$ of the KZ equation \eqref{eq:KZ}, 
we have the KZ equation
\begin{align}
 \frac{\p\tilde u}{\p x_i}&=\sum_{\substack{0\le\nu\le 3\\\nu\ne i}}\frac{\tilde A_{i,\nu}}{x_i-x_\nu}\tilde u\label{eq:KKZ}
\end{align}
satisfied by $\tilde u$.

Since this equation is reducible in general, we consider the reduced equation
\begin{align}
\bar{\mathcal M} : 
 \frac{\p\bar u}{\p x_i}&=\sum_{\substack{0\le\nu\le 3\\\nu\ne i}}\frac{\bar A_{i,\nu}}{x_i-x_\nu}\bar u.\label{eq:RKKZ}
\end{align}
The residue matrices $\tilde A_{i,j}$ and $\bar A_{i,j}$ are obtained 
from the results in \cite{DR, DR2, Ha}: %. Namely,
{\small%
\begin{equation}\label{eq:KxKZ}
\begin{split}
 \tilde A_{01}&=\begin{pmatrix}
  \mu+A_{01} & A_{02} & A_{03}+\lambda\\
 0 & 0 & 0\\
 0 & 0 & 0
\end{pmatrix},\qquad
\tilde A_{02}=\begin{pmatrix}
 0 & 0 & 0\\
 A_{01} &  \mu+A_{02} & A_{03}+\lambda\\
 0 & 0 & 0
\end{pmatrix},\\ %\\
\tilde A_{03}&=\begin{pmatrix}
 - \mu-\lambda & 0 & 0\\
 0 & - \mu-\lambda & 0\\
 A_{01} & A_{02} & A_{03}
\end{pmatrix},\qquad 
\tilde A_{04}=\begin{pmatrix}
 -A_{01}+\lambda & -A_{02} & -A_{03}-\lambda\\
 -A_{01} & -A_{02}+\lambda & -A_{03}-\lambda\\
 -A_{01} & -A_{02} & -A_{03}
\end{pmatrix},\\
\tilde A_{12}&=\begin{pmatrix}
 A_{12}+A_{02} & -A_{02} & 0\\
 -A_{01} & A_{12}+A_{01} & 0\\
 0 & 0 & A_{12}
\end{pmatrix},\ \ 
\tilde A_{13}=\begin{pmatrix}
 A_{13}+A_{03}+\lambda & 0 & -A_{03}-\lambda\\
 0 & A_{13} & 0\\
 -A_{01} & 0 & A_{01}+A_{13}
\end{pmatrix},\\
\tilde A_{14}&=
\begin{pmatrix}
 A_{23}-\mu-\lambda & 0 & 0\\
 A_{01} & A_{02}+A_{03}+A_{23} & 0\\
 A_{01} & 0 & A_{02}+A_{03}+A_{23}
\end{pmatrix},
\\
\tilde A_{23}&=
\begin{pmatrix}
  A_{23} & 0 & 0\\
 0 & A_{03}+A_{23}+\lambda & -A_{03}-\lambda\\
 0 & -A_{02} & A_{02}+A_{23}
\end{pmatrix},\allowdisplaybreaks\\
 \tilde A_{24}&=\begin{pmatrix}
 A_{01}+A_{13}+A_{03} & A_{02} & 0\\
 0 & A_{13}-\mu-\lambda &0\\
 0 & A_{02} & A_{01}+A_{13}+A_{03}
\end{pmatrix},\allowdisplaybreaks\\
 \tilde A_{34}&=\begin{pmatrix}
 A_{12}+A_{01}+A_{02}+ \mu& 0 & A_{03}+\lambda\\
 0 & A_{12}+A_{01}+A_{02}+ \mu & A_{03}+\lambda\\
 0 & 0 & A_{12}
\end{pmatrix}.%
\end{split}\end{equation}}
Here we denote $A_{01}=A_{0,1}$ etc.\ for simplicity.
%In general 
%the equation \eqref{eq:KKZ} may be reducible. In fact 

Then the subspace
\begin{align}
\begin{split}
 \mathcal L&:=\begin{pmatrix}\ker A_{01}\\ \ker A_{02}\\ \ker A_{03}+\lambda\end{pmatrix}
 +\ker(\tilde A_{04}-\mu-\lambda)\\
 &=\begin{pmatrix}\ker A_y\\ \ker A_1\\ \ker A_0+\lambda\end{pmatrix}
 +\ker\begin{pmatrix}A_y+\mu&A_1&A_0+\lambda\\ A_y&A_1+\mu &A_0+\lambda\\ A_y&A_1&A_0+\mu+\lambda\end{pmatrix}
\end{split}\end{align}
of $\mathbb C^{3N}$ satisfies $\tilde A_{i,j}\mathcal L\subset \mathcal L$.
We define $\bar A_{i,j}$ the square matrices of size $3N-\dim\mathcal L$ which correspond to linear transformations induced by $\tilde A_{i,j}$, respectively,  on the quotient space $\mathbb  C^{3N}/\mathcal L$.

It is known that if the equation \eqref{eq:ODE} is irreducible, then the corresponding ordinary differential equation defined by $\bar A_{0,1}$, $\bar A_{0,2}$ and $\bar A_{0,3}$ is 
irreducible (cf.~\cite{DR}) and so is the equation
\begin{align}
  \begin{cases}
   \displaystyle\frac{\p\bar u}{\p x}=
   \frac{\bar A_{01}}{x-y}\bar u+\frac{\bar A_{02}}{x-1}\bar u + \frac{\bar A_{03}}{x}\bar u,\\
   \displaystyle\frac{\p\bar u}{\p y}=
   \frac{\bar A_{01}}{y-x}\bar u+\frac{\bar A_{12}}{y-1}\bar u + \frac{\bar A_{13}}{y}\bar u.
  \end{cases}
\end{align}
%The statement above is a direct consequence of \cite{DR,Ha}.

Note that if $\lambda$ and $\mu$ are generic, we have
\begin{equation*}
 \mathcal L=\begin{pmatrix}
  \ker A_{01}\\
  \ker A_{02}\\
  0
 \end{pmatrix}.
\end{equation*}

Next we examine the transformations
\begin{align}
\begin{split}
\tilde K_{y}^{\mu,\lambda}&:=T_{(x,y)\mapsto(y,x)}\circ \tilde K_x^{\mu,\lambda}\circ T_{(x,y)\mapsto(y,x)},\\
\tilde K_{x,y}^{\mu,\lambda}&:=T_{(x,y)\mapsto(x,\frac xy)}\circ \tilde K_x^{\mu,\lambda}\circ T_{(x,y)\mapsto(x,\frac xy)}.
\end{split}
\end{align}
 Note that $(x,y)\mapsto (y,x)$ and $(x,y)\mapsto (x,\frac xy)$ correspond to 
$(x_0,x_1,x_2,x_3,x_4)\mapsto(x_1,x_0,x_2,x_3,x_4)$ and 
$(x_0,x_1,x_2,x_3,x_4)\mapsto(x_2,x_1,x_0,x_4,x_3)$, respectively.
Hence the KZ equations satisfied by $\tilde K_y^{\mu,\lambda}u$ and $\tilde K_{x,y}^{\mu,\lambda}u$ are easily 
obtained from their definition and the equation satisfied by $\tilde K_x^{\mu,\lambda}u$. 
We  consider the equation satisfied by $\tilde K_{x,y}^{\mu,\lambda}u$. 

Putting
\begin{equation}\label{eq:I2KZ}
\tilde u(x,y)=(\tilde K_{x,y}^{\mu,\lambda}u)(x,y) 
 =\left(\begin{smallmatrix}
  T_{(x,y)\mapsto(x,\frac xy)} K_{x}^{\mu+1,\lambda}\frac x{x-y}u(x,\frac xy)\\
  T_{(x,y)\mapsto(x,\frac xy)} K_{x}^{\mu+1,\lambda}\frac x{x-1}u(x,\frac xy)\\
  T_{(x,y)\mapsto(x,\frac xy)} K_{x}^{\mu+1,\lambda}u(x,\frac xy)
\end{smallmatrix}
\right), 
\end{equation}
the residue matrices of KZ equation satisfied by $\tilde u(x,y)$ are given by
{\small\begin{equation}\label{eq:KxyKZ}
\begin{split}
%%%
%%
 &\tilde A_{01}=\begin{pmatrix}
 A_{01}+A_{02}&-A_{02}& 0 \\
 -A_{12}&A_{01}+A_{12}& 0 \\
0& 0&A_{01} 
\end{pmatrix},\quad
\tilde A_{02}=\begin{pmatrix}
0& 0& 0 \\
 A_{12}& A_{02}+\mu&A_{24}+\lambda \\
0& 0& 0 
\end{pmatrix},\\
&\tilde A_{03}=\begin{pmatrix}
 A_{03}&A_{02}&0\\
 0&A_{14}-\mu-\lambda&0\\
 0&A_{02}&A_{03}
\end{pmatrix},\quad
 \tilde A_{04}
=\begin{pmatrix}
 A_{04}+A_{24}& 0&0 \\
0&A_{04}+A_{24}+\lambda& -A_{24}-\lambda \\
0& -A_{02}&A_{02}+A_{04}
\end{pmatrix},\\
&\tilde A_{12}=\begin{pmatrix}
  A_{12}+\mu&A_{02}&A_{24}+\lambda \\
0& 0& 0 \\
0& 0& 0 
\end{pmatrix},\quad 
\tilde A_{13}=
\begin{pmatrix}
 A_{04}-\mu-\lambda&0 & 0\\
 A_{12}&A_{13}& 0 \\
 A_{12}&0&A_{13} 
\end{pmatrix},\\
&\tilde A_{14}=\begin{pmatrix}
 A_{14}+A_{24}+\lambda& 0& -A_{24}-\lambda \\
0&A_{14}+A_{24}& 0 \\
-A_{12}& 0&A_{12}+A_{14} 
\end{pmatrix},\\
&\tilde A_{23}=\begin{pmatrix}
 -A_{12}+\lambda&-A_{02}&-A_{24}-\lambda \\
 -A_{12}&-A_{02}+\lambda&-A_{24}-\lambda \\
 -A_{12}&-A_{02}&-A_{24} 
\end{pmatrix},\quad
\tilde A_{24}=\begin{pmatrix}
 - \mu-\lambda& 0& 0 \\
0&- \mu-\lambda& 0 \\
 A_{12}&A_{02}&A_{24} 
\end{pmatrix},\\
%%%
 &\tilde A_{34}=\begin{pmatrix}
 A_{01}+A_{02}+A_{12}+\mu & 0 & A_{24}+\lambda\\
 0 & A_{01}+A_{02}+A_{12}+ \mu& A_{24}+\lambda\\
 0 & 0 & A_{01}
\end{pmatrix}%
\end{split}
\end{equation}
and the invariant subspace to define the required residue matrices $\bar A_{i,j}$ is
\begin{align}\begin{split}
\mathcal L&=\begin{pmatrix}\ker A_{12}\\ \ker A_{02}\\ \ker A_{24}+\lambda\end{pmatrix}
 +\ker (\tilde A_{23}-\mu-\lambda)\\
 &=\begin{pmatrix}\ker B_1\\ \ker A_1\\ \ker A_{24}+\lambda\end{pmatrix}+\ker
 \begin{pmatrix}B_1+\mu&A_1&A_{24}+\lambda\\ B_1&A_1+\mu&A_{24}+\lambda\\ B_1&A_1&A_{24}+\mu+\lambda \end{pmatrix}
 \subset \mathbb C^{3N}.
\end{split}\end{align}}

Lastly in this section we give an example of hypergeometric series characterized 
by a KZ equation.  Namely, 
applying 
\begin{equation}\label{eq:Kpqr}
  \prod_{i=2}^p\tilde K_x^{-\alpha'_i-\alpha_i,\alpha_i} 
 \prod_{j=2}^q\tilde K_y^{-\beta'_j-\beta_j,\beta_j} \prod_{r=1}^r\tilde K_{x,y}^{-\gamma'_k-\gamma_k,\gamma_k}
\end{equation}
to a solution of the equation
$\D u=\alpha_1u \frac{\D x}{x-1}+\beta_1u \frac{\D y}{y-1}$, we get a KZ 
equation \eqref{eq:KZxy} with
the generalized Riemann scheme (see \cite[\S4]{Okz} for its definition)
\begin{align}\label{eq:RSKZ}
\begin{split}&\left\{\begin{matrix}
A_{01}&A_{02}&A_{03}&A_{04}&A_{12}\\
[0]_{pq+(p+q-1)r} & [0]_{pr+(p+r-1)q} & [\alpha'_i]_{q+r} & [\alpha_i]_{q+r} & [0]_{qr+(q+r-1)p} \\
[-\alpha''-\beta'']_r & [-\alpha''-\gamma'']_q & \beta_j+\gamma'_k &  \beta'_j+\gamma_k& [-\beta''-\gamma'']_p \\
%&  & & &  \\
%&  & & &  \\
%&  & & &  
\end{matrix}\right.\\
&\qquad
 \left.\begin{matrix}
A_{13}&A_{23}&A_{14}&A_{24}&A_{34}\\
 [\beta'_j]_{p+r} & [\gamma_k]_{p+q} & [\beta_j]_{p+r} & [\gamma'_k]_{p+q} & [0]_{pq+qr+rp-(p+q+r)+1}\\
\alpha_i+\gamma'_k & \alpha_i+\beta_j & \alpha_i'+\gamma_k & \alpha'_i+\beta'_j& [-\alpha''-\beta''-\gamma'']_2\\
 & & & & [-\alpha'' - \beta'']_{r-1}\\
 & & & &[-\beta''-\gamma'']_{p-1}\\
 & & & & [-\alpha''-\gamma'']_{q-1}\\
\end{matrix}\right\},
\end{split}
%\intertext{with}
\allowdisplaybreaks\\
\begin{split}
&\qquad\qquad
 \alpha''_i:=\alpha_i+\alpha'_i,\ \beta''_j:=\beta_j+\beta'_j,\ \gamma''_k:=\gamma_k+\gamma'_k,\
\alpha'_1=\beta'_1=0,\\
&\qquad\qquad
 \alpha''=\sum_{i=1}^p\alpha''_i,\  \beta''=\sum_{j=1}^q\beta''_j,\ 
 \gamma''=\sum_{k=1}^r\gamma''_k,\\
&\qquad\qquad 1\le i\le p,\ 1\le j\le q,\ 1\le k\le r\quad(p\ge1,\ q\ge1,\ r\ge1).
\end{split}
\end{align}
Here, for example, the eigenvalues of the square matrix $A_{01}$ of size $R=pq+qr+rp$
are 0 with multiplicity $pq+(p+q-1)r$ and $-\alpha''-\beta''$ with multiplicity $r$. 
If the parameters $\alpha_i$, $\beta_j$, $\gamma_k$, $\alpha'_i$, $\beta'_j$, $\gamma'_k$ are generic, the matrices $A_{i,j}$ are semisimple and the KZ equation is irreducible.

Note the hypergeometric series
\begin{equation}\label{eq:HGpqr}
\begin{split}
\phi(x,y)&=\sum_{m=0}^\infty\sum_{n=0}^\infty
  \frac{\prod_{i=1}^p (\alpha_i)_{m}\prod_{j=1}^q (\beta_j)_{n}\prod_{k=1}^r (\gamma_k)_{m+n}}
  {\prod_{i=1}^{p} (1-\alpha'_i)_{m}\prod_{j=1}^{q} (1-\beta'_j)_n\prod _{k=1}^{r}(1-\gamma'_k)_{m+n}}x^my^n\\
 &\qquad
\text{with}\ \ \alpha'_1=\beta'_1=0
\end{split}
\end{equation}
is a component of a solution of this KZ equation (cf.~\eqref{eq:IKZ}).

The Riemann scheme \eqref{eq:RSKZ} is obtained by \cite[Theorem~7.1]{Okz} and \eqref{eq:Kpqr}.
The precise argument and a further study of the hypergeometric series \eqref{eq:HGpqr} will 
be given in another paper.

The index of the rigidity of this KZ equation with respect to $x$ equals
\begin{equation}
  \begin{split}
  \mathrm{Idx}_x\mathcal M&=(R-q)^2+q^2+(R-r)^2+r^2+2(p(q+r)^2 +qr)-2R^2\\
  &=2-2(q-1)(r-1)(q+r+1)
\end{split}
\end{equation}
and hence the ordinary differential equation with respect to the variable $x$ is rigid if and only 
if $r=1$ or $q=1$.

If $p=q=r=1$,  the corresponding KZ equation \eqref{eq:KKZ} is given by \eqref{eq:KxyKZ} with
\begin{align*}
 &(A_{01},A_{02},A_{03},A_{04},A_{12},A_{13},A_{14},A_{23},A_{24},A_{34},\lambda,\mu)\\
 &\quad=
 (0,\alpha_1,0,-\alpha_1,\beta_1,0,-\beta_1,-\alpha_1-\beta_1,\alpha_1+\beta_1,\gamma_1,-\gamma_1-\gamma'_1)
\end{align*}
and if follows from \eqref{eq:I2KZ} that the equation has a solution with the last component
\[\phi(x,y)=F_1(\gamma_1,\alpha_1,\beta_1,1-\gamma'_1;x,y).\]
%%%

We define a simple local solution to \eqref{eq:KZxy} at the origin, which includes the solution
we have just considered.
%%%
\begin{definition}\label{def:simple}
We define that a local solution to the equation \eqref{eq:KZxy} near the origin have
a \emph{simple monodromy} if the analytic continuation of the solution in a 
neighborhood of the origin spans one dimensional space.  
We simply call the solution a \emph{simple} solution at the origin.
We also define that a local solution of the equation to \eqref{eq:KZxy} near the line $x=0$ have a 
simple monodromy and call it a simple solution along $x=0$ 
if the analytic continuation of the solution in a neighborhood of $x=0$ spans one dimensional 
space.

By the correspondence between the equations  \eqref{eq:KZxy} and \eqref{eq:KZ} with $q=3$
and moreover a transformation by $\mathfrak S_5$ we define a local solution at $x_i=x_j=x_k$ 
and a local solution at $x_{s}=x_{t}$  to the equation  \eqref{eq:KZ} with $q=3$ 
when $\{i,j,k,s,t\} =\{0,1,2,3,4\}$.
\end{definition}

Here, for example, the path of the analytic continuation in the latter case, namely along $x=0$,  
is in $\{(x,y)\in\mathbb P^1\times\mathbb P^1
\mid |x|<\epsilon, \ 0 < |x| < \epsilon |y|\}$
with a small positive number $\epsilon$.

%%%
Then we have the following theorem.
\begin{theorem}\label{thm:KZ2}
Suppose  $\{i,j,k,s,t\} =\{0,1,2,3,4\}$ as above.
To the equation \eqref{eq:KZ} with $q=3$ 
there is one to one correspondence between a simle local at $x_i=x_j=x_k$ and a simple 
solution along $x_s=x_t$.
\end{theorem}
\begin{proof}

The coordinate $(x,\frac xy)$ is a local coordinate of a blowing up of the singularities of the
equation \eqref{eq:KZ} around the origin.  Then the origin corresponds to the line $x=0$.
This coordinate transformation corresponds to the map $(x_0,x_1,x_2,x_3,x_4)$ $\mapsto
(x_2,x_1,x_0,x_4,x_3)$, which is explained in Remark~\ref{rem:blowup}.
 Since $x_2=x_4$ corresponds to $x_0=x_3$, we have the theorem when
$(i,j,k,s,t)=(0,1,3,2,4)$.  Note that the coordinate $(\frac yx,y)$ gives the same conclusion.
Then the symmetry $\mathfrak S_5$ proves the theorem.
\end{proof}
%%%
This theorem says that an eigenvalue of $A_{24}$ with free multiplicity corresponds to a 
simple solution at $x_0=x_1=x_3$ which is the origin in $(x,y)$ coordinate.
Hence at the origin we have $pq$ independent simple solutions of the KZ equation with the Riemann 
scheme \eqref{eq:RSKZ} if the parameters are generic.
%%%
\begin{remark}  The space of local solutions at a normally 
crossing singular point defined by $x_i=x_j$ and $x_s=x_t$ under the above notation 
is spanned by simple solutions if the parameters are generic (cf.~\cite{KO}).

\end{remark}
%%%
\begin{remark}\label{rm:1-x}
The transformation
\begin{equation*}
  \mathcal O_0\ni u=\sum c_{m,n}x^my^n \mapsto (1-x)^{-a}(1-y)^{-b}u
 =\sum (a)_i(b)_jc_{m,n}x^{i+m}y^{j+n}
\end{equation*}
induces the transformation $(A_{02},A_{12})\mapsto (A_{02}+a,A_{12}+b)$ of the equation 
\eqref{eq:KZxy}.
Then the coefficients of the resulting power series may be complicated (cf.~\eqref{eq:211}).
\end{remark}
%%%
\begin{remark}
The transformation of residue matrices induced by $\tilde K_x^\lambda$, $\tilde K_y^\lambda$ and  
the $\tilde K_{x,y}^\lambda$ and the calculation of 
Riemann scheme \eqref{eq:RSKZ} for given $p$, $q$ and $r$ etc.\ are 
supported by the functions \texttt{m2mc} and \texttt{mc2grs} in the library \cite{ORisa}
of the computer algebra \texttt{Risa/Asir}.
For example, by the commands
\begin{verbatim}
    R=os_md.mc2grs(0,["K",[4,3,2]]);
    os_md.mc2grs(R,"get"|dviout=1,div=5);
\end{verbatim}
we get \eqref{eq:RSKZ} on a display in the case $(p,q,r)=(4,3,2)$.
This is enabled by using a \texttt{PDF} file output by functions in the library
under the computer algebra.
Here \texttt{div=5} indicates to divide the Riemann scheme by 5 columns into two parts as 
in \eqref{eq:RSKZ} and \texttt{R} is a list of simultaneous eigenspace decomposition at 
15 normally crossing singularities of the corresponding KZ equation.
The algorithm for the calculation is given by \cite[Theorem~7.1]{Okz}.
Moreover by the command
\begin{verbatim}
    os_md.mc2grs(R,"rest"|dviout=1);
\end{verbatim}
we get the Riemann scheme of the induced equations on the 10 singular hypersurfaces 
corresponding to eigenvalues of 10 residue matrices.
If \texttt{"spct"} is indicated in place of \texttt{"rest"}, a table of 
spectral types with respect to the variables $x_i$ for $i=0,\dots,4$ and the indices of rigidities
are displayed.  If ``\texttt{dviout=-1}" is indicated in place of \texttt{dviout=1}, the result is 
given by a \TeX\ source in place of displaying the result on a screen. If ``\texttt{|dviout=1}" is not indicated, the result is given in a format recognized by \texttt{Risa/Asir}.
\end{remark}
%%%%%%%%%%%%%%%%%%%%%%%%%%%%%%%%%%%
%%%%%%%%%%%%%%%%% ODE %%%%%%%%%%%%%
%%%%%%%%%%%%%%%%%%%%%%%%%%%%%%%%%%%
\section{Fuchsian ordinary differential equations}
In this section we consider a Fuchsian differential equation
\begin{align}\label{eq:OdKZ}
\mathcal N : \frac {\D u}{\D x}=\frac{A_y}{x-y}u+\frac{A_1}{x-1}u+\frac{A_0}xu
\end{align}
with regular singularities at $x=0$, $1$, $y$ and $\infty$. 
Here the residue matrices $A_y$, $A_1$ and $A_0$ are constant square matrices of size $N$.
If \eqref{eq:OdKZ} is irreducible and rigid or
\begin{gather}
 \dim\bigl(Z(A_y)\cap Z(A_1)\cap Z(A_0)\bigr)=1\label{eq:IrrKZ}
 \intertext{ and }
 \dim Z(A_y)+\dim Z(A_1)+\dim Z(A_0)+\dim Z(A_y+A_1+A_0) - 2N^2=2,\label{eq:idx}
\end{gather}
the equation
\eqref{eq:OdKZ} is constructed by successive 
applications of middle convolutions and additions to the trivial equation $u'=0$.
Here $Z(A)$ denotes the space of the centralizer in $M(N,\mathbb C)$ for $A\in M(N,\mathbb C)$
and the left hand side of \eqref{eq:idx} is the index of the rigidity of the equation.
We note that middle convolutions can be replaced by the transformation of equations 
induced by $\tilde K_x^\mu$ with additions.

We assume that the equation \eqref{eq:OdKZ} can be extended to the compatible equation
\begin{align} 
\frac {\p u}{\p y}=\frac{A_y}{x-y} u+\frac{B_1}{y-1}u+\frac{B_0}y u.
\end{align}
If the equation \eqref{eq:OdKZ} is rigid, it extends to the compatible equation and moreover if 
the equation satisfies \eqref{eq:IrrKZ}, 
the matrices $B_0$ and $B_1$ are uniquely determined by \eqref{eq:OdKZ} up to the difference of
scalar matrices. 
We can apply the transformations induced by $\tilde K_y^\mu$ and $\tilde K_{x,y}^\mu$
to \eqref{eq:OdKZ}.
Note that these transformations may change the index of rigidity as was shown in the last section.

The transformations induced by $\tilde K_x^\mu$, $\tilde K_y^\mu$ and $\tilde K_{x,y}^\mu$ are 
given by the following \eqref{eq:OdKx}, \eqref{eq:OdKy} and \eqref{eq:OdKxy}, respectively, 
with calculating the induced matrices of the residue matrices on $\mathbb C^{3N}/\mathcal L$.
\begin{align}
\label{eq:OdKx}
\begin{split}
\frac {\D\tilde u}{\D x}
&=
 \frac{\left(\begin{smallmatrix}
  A_y+\mu&A_1&A_0\\ 0&0&0\\ 0&0&0
 \end{smallmatrix}\right)}{x-y}\tilde u
 +
 \frac{\left(\begin{smallmatrix}
  0 & 0 & 0\\
  A_y&A_1+\mu&A_0\\
  0&0&0
 \end{smallmatrix}\right)}{x-1}\tilde u
 +
 \frac{\left(\begin{smallmatrix}
  -\mu & 0 & 0\\
  0& -\mu & 0\\
  A_y&A_1&A_0
 \end{smallmatrix}\right)}{x}\tilde u,\\
\mathcal L&=\left(\begin{smallmatrix}\ker A_y\\ \ker A_1\\ \ker A_0\end{smallmatrix}\right)
 +\ker\left(\begin{smallmatrix}A_y+\mu&A_1&A_0\\ A_y&A_1+\mu &A_0\\ A_y&A_1&A_0
  +\mu\end{smallmatrix}\right),
\end{split}
\allowdisplaybreaks\\
%%%%
\label{eq:OdKy}
\begin{split}
\frac {\D\tilde u}{\D x}
&=
 \frac{\left(\begin{smallmatrix}
  A_y+\mu&B_1&B_0\\ 0&0&0\\ 0&0&0
 \end{smallmatrix}\right)}{x-y}\tilde u
 +
 \frac{\left(\begin{smallmatrix}
  A_1+B_1 & -B_1 & 0\\
  -A_y &A_1+A_y& 0\\
  0&0&A_1
 \end{smallmatrix}\right)}{x-1}\tilde u
 +
 \frac{\left(\begin{smallmatrix}
  A_0+B_0&0 & -B_0\\
  0& A_0& 0\\
  -A_y& 0&A_0+A_y
 \end{smallmatrix}\right)}{x}\tilde u,
\\
\mathcal L&=\left(\begin{smallmatrix}\ker A_y\\ \ker B_1\\ \ker B_0\end{smallmatrix}\right)
 +\ker\left(\begin{smallmatrix}A_y+\mu&B_1&B_0\\ B_y&A_1+\mu &B_0\\ B_y&A_1&B_0+\mu\end{smallmatrix}\right),
\end{split}
\allowdisplaybreaks\\
%%%
\label{eq:OdKxy}
\begin{split}
\frac {\D\tilde u}{\D x}
&=
 \frac{\left(\begin{smallmatrix}
  A_y+A_1&-A_1&0\\ -B_1&A_y+B_1&0\\ 0&0&A_y
 \end{smallmatrix}\right)}{x-y}\tilde u
 +
 \frac{\left(\begin{smallmatrix}
  0 & 0 & 0\\
  B_1&A_1+\mu&A_{24}\\
  0&0&0
 \end{smallmatrix}\right)}{x-1}\tilde u
 +
 \frac{\left(\begin{smallmatrix}
  A_0 & A_1 & 0\\
  0& A_{14}-\mu & 0\\
  0&A_1&A_0
 \end{smallmatrix}\right)}{x}\tilde u,\\
 \mathcal L&=\left(\begin{smallmatrix}\ker B_1\\ \ker A_1\\ \ker A_{24}
 \end{smallmatrix}\right)
 +\ker \left(\begin{smallmatrix}B_1+\mu&A_1&A_{24}\\ 
 B_1&A_1+\mu&A_{24}\\ B_1&A_1&A_{24}+\mu \end{smallmatrix}\right).
\end{split}
\end{align}
Here 
\begin{align*}
A_{14}&=-A_y-B_0-B_1,\allowdisplaybreaks\\
A_{24}&=-(A_{02}+A_{12}+A_{23})=(A_{01}+A_{02}+A_{03}+A_{12}+A_{13})-A_{02}-A_{12}\\
 & = A_{01}+A_{03}+A_{13}= A_y+A_0+B_0.
\end{align*}

\if\ORG
\else
\begin{acknowledgement}
This work was supported by Grant-in-Aid for Scientific Researches (C), 
No.\ 18K03341, Japan Society of Promotion of Science.
\end{acknowledgement}
\fi


\small
\begin{thebibliography}{GKZ}
\bibitem[AK]{AK}
Appell K.\ and J.\ Kamp\'e de F\'eriet,
\textit{Fonctions hyperg\'eom\'etriques et hypersph\'eriques polynomes
d'Hermite}, Gauthier-Villars, 1926.

\bibitem[DR]{DR} 
M.~Dettweiler and S.~Reiter, 
An algorithm of Katz and its applications
to the inverse Galois problems, 
J.\ Symbolic Comput.\ 
\textbf{30} (2000), 761--798.

\bibitem[DR2]{DR2}
Dettweiler M.\ and S.~Reiter,
Middle convolution of Fuchsian systems and the construction of
rigid differential systems,
J.\ Algebra 
\textbf{318} (2007), 1--24. 

\bibitem[Er]{Er}
Erd\'elyi A.,\ W.\ Magnus, F.\ Oberhettinger and F.\ G.\ Tricomi,
\textit{Higher Transcendental Functions}, 3 volumes, McGraw-Hill Book Co.,
New York, 1953.

\bibitem[GKZ]{GKZ}
Gel'fand I.~M., M.~M.~Kpranov and A.~V.~Zelevinsky,
Generalized Euler integrals and $A$-hypergeometric functions,
Adv.\ Math.\ \textbf{84} (1990), 255--271.

\bibitem[Ha]{Ha}
Haraoka, Y.,
Middle convolution for completely integrable systems with logarithmic
singularities along hyperplane arrangements,
Adv.\ Studies in Pure Math.\ \textbf{62}(2012), 109--136.

%\bibitem[Ha]{Ha}
%Y.~Haraoka,
%\textit{On Oshima's middle convolution},
%Josai Mathematical Monographs {\bf 12} (2020),
%19--51.

%\bibitem[Hi]{Hi}
%K.\ Hiroe,
%\textit{Linear differential equations on $\mathbb P^1$ and root systems}, 
%J. Algebra \textbf{382} (2013), 1--38.

\bibitem[Ho]{Ho}
Horn J.,
\"Uber die Convergenz der hypergeometrischen Reihen zweier und dreier Ver\"anderlichen,
Math.\ Annalen \textbf{34} (1889), 544--600.
%\textit{Hypergeometrische Funktionen zweier Veranderlichen}, Math.\ Annalen \textbf{105} (1931),
%381--407.

\bibitem[KO]{KO}
Kashiwara M.\ and T.~Oshima,
Systems of differential equations with regular singularities and their
boundary value problems,
Annal.\ of Math., \textbf{106} (1977), 145--200.

\bibitem[Ka]{Kz}
Katz N.~M., 
\textit{Rigid Local Systems}, Annals of Mathematics Studies \textbf{139},
Princeton University Press 1995, doi: 10.1515/9781400882595.

\bibitem[KZ]{KZ}
Knizhnik K.\ and A.~Zamolodchikov, 
Current algebra and Wess-Zumino model in 2 dimensions,
Nucl.\ Phys.\ B \textbf{247} (1984), 83--103.  

\bibitem[La]{La}
Lauricella G.,
Sulle funzioni ipergeometriche a piu variabili, 
Rendiconti del Circolo Matematico di Palermo.\ \textbf{7} (1983), 111--158.

\bibitem[Ma]{Ma} %No
Matsubara-Hao S.-J.,
Global analysis of GG systems, Int.\ Math.\ Res.\ Not.\ \textbf{2022.19} (2022), 
14923--14963, doi: 10.1093/imrn/rnab144

\bibitem[O1]{Ow}
Oshima T.,
\textit{Fractional calculus of Weyl algebra and Fuchsian differential equations},
MSJ Memoirs \textbf{28}, Mathematical Society of Japan, Tokyo, 2012.

%\bibitem[O2]{Oirr}
%T.~Oshima,
%\textit{Reducibility of hypergeometric equations}, 
%Analytic, Algebraic and Geometric Aspects of
%Differential Equations, Trends in Mathematics, Springer, 2017, 429--453.

\bibitem[O2]{Ogauss} %No
Oshima T.,
An elementary approach to the Gauss hypergeometric function,
Josai Mathematical Monographs {\bf 6} (2013)
3--23, doi: 10.20566/1344777-06-3.

\bibitem[O3]{Okz}
Oshima T.,
Transformation of KZ type equations, 
Microlocal Analysis and Singular Perturbation Theory,
RIMS K\^oky\^uroku Bessatsu \textbf{B61} (2017), 141--162.

%\bibitem[O4]{Ol} 
%T.~Oshima, \textit{Semilocal monodromy of rigid local system}, 
%in \emph{Formal and Analytic Solutions of Diff.~Equations}, 
%Springer Proceedings in Mathematics and Statics \textbf{256} (2018), 189--199.

\bibitem[O4]{Ojm}
Oshima T., 
Confluence and versal unfolding of Pfaffian systems, 
Josai Mathematical Monographs {\bf 12} (2020),
 117--151,  
doi: 10.20566/13447777-12-117.

\bibitem[O5]{Ov} %No
Oshima T., 
Versal unfolding of irregular singularities
of a linear differential equation on the Riemann sphere,
Publ.~RIMS Kyoto Univ.\ \textbf{57} (2021), 893--920, 
doi: 10.4171/PRIMS/57-3-6.

\bibitem[O6]{Ori} %No
Oshima T., 
Riemann-Liouville transform and linear differential
equations on the Riemann sphere, 
Recent Trends in Formal and Analytic Solutions of Diff.\ Equations,
Contemporary Mathematics \textbf{782} (2023), 57--91, 
American Mathematicasl Society.

\bibitem[O7]{ORisa}
Oshima T.,
\texttt{os\_muldif.rr}, a library of computer algebra \texttt{Risa/Asir},
2008--2023.\\
\if\ORG
\texttt{http://www.ms.u-tokyo.ac.jp/\~{}oshima/}
\else
\url{http://www.ms.u-tokyo.ac.jp/\~{}oshima/}
\fi

%\bibitem[WW]{WW}
%E.~T.~ Whittaker and G.~N.~Watson,
%A Course of Modern Analysis, 4th edition,
%Cambridge University Press, London, 1955.
\end{thebibliography}
\end{document}